\documentclass[letterpaper,12pt]{article}
\usepackage{fullpage}		
\usepackage{amsfonts}
\usepackage{bm}

\usepackage{amssymb}
\usepackage{amsmath}
\usepackage{amsthm}
\usepackage{color}
\usepackage{mathtools} 

\usepackage{graphicx}

\usepackage{xfrac} 

\usepackage{float}		

\usepackage[colorlinks=true, citecolor=blue, urlcolor=cyan, linktocpage]{hyperref}

\usepackage{enumerate}	

\usepackage{tikz}				
\usetikzlibrary{arrows,matrix,positioning}

\usepackage{wrapfig}

\usepackage{cancel} 		

\usepackage{verbatim}   

\usepackage{centernot}  

\usepackage{mdframed}		

\definecolor{shadecolor}{gray}{0.90}				
\def\boitegrise#1#2{\begin{centerline}{\fcolorbox{black}{shadecolor}{~
    \begin{minipage}[t]{#2}{\vphantom{~}#1\vphantom{$A_{\displaystyle{A_A}}$}}
            \end{minipage}~}}\end{centerline}\medskip}

\DeclareMathOperator{\odd}{\mathit{Odd}}
\DeclareMathOperator{\even}{\mathit{Even}}
\DeclareMathOperator{\im}{im}
\DeclareMathOperator{\zrho}{\mathbb{Z}[\rho]}
\DeclareMathOperator{\ephi}{\varphi_{\!\rho}}

\newcommand{\bigslant}[2]{{\raisebox{.15em}{$#1$}\left/\raisebox{-.15em}{$#2$}\right.}}
\newcommand{\zrhomod}[1]{\bigslant{\zrho}{(#1)}}
\newcommand{\zrhomodunit}[1]{\left(\bigslant{\zrho}{(#1)}\right)^\times}
\newcommand{\divides}{\,\vert\,}
\newcommand{\ndivides}{\centernot\vert}
\newcommand{\plus}{\texttt{+}}
\newcommand{\minus}{\raisebox{.7pt}{-}}

\newtheorem{theorem}{Theorem}[section] 
\newtheorem{lemma}[theorem]{Lemma}
\newtheorem{corollary}[theorem]{Corollary}

\newtheorem{proposition}[theorem]{Proposition}

\theoremstyle{definition}
\newtheorem{example}[theorem]{Example}

\theoremstyle{definition}
\newtheorem{definition}[theorem]{Definition}

\theoremstyle{plain}

\theoremstyle{definition}
\newtheorem{remark}[theorem]{Remark}

\theoremstyle{definition}
\newtheorem{convention}[theorem]{Convention}

\theoremstyle{definition}
\newtheorem{question}[theorem]{Question}

\numberwithin{figure}{section}
\numberwithin{theorem}{section}

\title{\textcolor{red}{\textbf{An Euler phi function for the Eisenstein integers and some applications}}}

\author{
	Emily Gullerud\\
	{\small\textit{University of Minnesota}}\\
  {\small\href{mailto:gulle069@umn.edu}{\nolinkurl{gulle069@umn.edu}}}
  \and
  aBa Mbirika\\
	{\small\textit{University of Wisconsin-Eau Claire}}\\
  {\small\href{mailto:mbirika@uwec.edu}{\nolinkurl{mbirika@uwec.edu}}}
}

\makeatletter
\newcommand\ackname{Acknowledgments}
\if@titlepage
  \newenvironment{acknowledgments}{%
      \titlepage
      \null\vfil
      \@beginparpenalty\@lowpenalty
      \begin{center}%
        \bfseries \ackname
        \@endparpenalty\@M
      \end{center}}%
     {\par\vfil\null\endtitlepage}
\else
  \newenvironment{acknowledgments}{%
      \if@twocolumn
        \section*{\abstractname}%
      \else
        \small
        \begin{center}%
          {\bfseries \ackname\vspace{-.5em}\vspace{\z@}}%
        \end{center}%
        \quotation
      \fi}
      {\if@twocolumn\else\endquotation\fi}
\fi
\makeatother

\begin{document}

\maketitle

\begin{abstract}
The Euler phi function on a given integer $n$ yields the number of positive integers less than $n$ that are relatively prime to $n$. Equivalently, it gives the order of the group of units in the quotient ring $\bigslant{\mathbb{Z}}{(n)}$ for a given integer $n$. We generalize the Euler phi function to the Eisenstein integer ring $\zrho$ where $\rho$ is the primitive third root of unity $e^{2\pi i/3}$ by finding the order of the group of units in the ring $\zrhomod{\theta}$ for any given Eisenstein integer $\theta$. As one application we investigate a sufficiency criterion for when certain unit groups $\zrhomodunit{\gamma^n}$ are cyclic where $\gamma$ is prime in $\zrho$ and $n \in \mathbb{N}$, thereby generalizing well-known results of similar applications in the integers and some lesser known results in the Gaussian integers. As another application, we prove that the celebrated Euler-Fermat theorem holds for the Eisenstein integers.

\end{abstract}

\tableofcontents 



\section{Introduction}

It is well known that the integers $\mathbb{Z}$, the Gaussian integers $\mathbb{Z}[i]$, and the Eisenstein integers $\zrho$ where $\rho = e^{2 \pi i/3}$ share many structural properties. For instance, each ring is a Euclidean domain and hence also a principal ideal domain and unique factorization domain as a consequence. Thus many results over the integers have been extended to both the Gaussian and Eisenstein integers. However one primary difference is that while $\mathbb{Z}$ has a total ordering, there are no such orderings on $\mathbb{Z}[i]$ or $\zrho$. Yet it is still possible to extend the Euler phi function on $\mathbb{Z}$ to $\mathbb{Z}[i]$ or $\zrho$. The Euler phi function $\varphi$ evaluated on an integer $n \geq 1$  yields the number of positive integers less than $n$ that are relatively prime to $n$. Observe that an alternative (yet equivalent) way to view the image $\varphi(n)$ of $n \in \mathbb{Z}$ arises from considering the number of units in the multiplicative group $\bigslant{\mathbb{Z}}{(n)}$ as follows:
\begin{align*}
\varphi(n)		&= \left| \{k \in \mathbb{N} \, : \, k<n \mbox{ and } \gcd(k,n) = 1 \} \right| \\
					&= \left| \{ [k] \in \bigslant{\mathbb{Z}}{(n)} \, : \, \gcd(k,n) = 1 \} \right| \\
					&= \left| \left( \bigslant{\mathbb{Z}}{(n)} \right)^\times \right|.
\end{align*}
Although no ordering exists on the Gaussian or the Eisenstein integers, for any element $r$ in a ring $R$ we can consider the corresponding quotient $\bigslant{R}{(r)}$ and find the cardinality of its unit group $\left( \bigslant{R}{(r)} \right)^\times$. Indeed this has been done already in 1983 by Cross for the Gaussian integers~\cite{Cross1983}. In this paper we extend Cross's idea to the Eisenstein integers by doing the following:
\begin{enumerate}
\item We present the well-known characterization of the three types of primes in $\zrho$ (see Proposition~\ref{prop:primes}).
\item For each prime $\gamma \in \zrho$ and $n \in \mathbb{N}$, we define a complete set of residue classes for the quotient $\zrhomod{\gamma^n}$ (see Theorem~\ref{thm:equivClasses}). Moreover, we identify the units in this quotient (see Theorem~\ref{thm:unitsInRings}).
\item Analogous to the $\mathbb{Z}$-setting, we define an Euler phi function $\ephi$ on the Eisenstein integers to be the size of the corresponding groups of units. That is, given $\eta \in \zrho$, then $\ephi(\eta) = \Big|\zrhomodunit{\eta}\Big|$ (see Section~\ref{sec:Eulerphi}).
\item For a prime $\gamma \in \zrho$ and an integer $n \geq 1$, we compute the value $\ephi (\gamma^n)$ (see Theorem~\ref{thm:sizeOfUnits}). Furthermore we prove $\ephi$ is multiplicative (see Corollary~\ref{cor:phi_is_mult}).
\item Exploiting the fact that $\zrho$ is a unique factorization domain, we may write $\theta \in \zrho$ uniquely as a product of powers of primes $\theta = \gamma_1^{k_1} \gamma_2^{k_2} \cdots \gamma_r^{k_r}$ up to associates. Then the Euler phi function on $\theta$ is
$$ \ephi(\theta) = \prod_{i=1}^r \ephi \left( \gamma_i^{k_i} \right).$$
\end{enumerate}

We conclude the paper by giving applications of the Euler phi function on $\zrho$. As one application we investigate a sufficiency criterion for when certain unit groups $\zrhomodunit{\gamma^n}$ are cyclic where $\gamma$ is prime in $\zrho$ and $n \in \mathbb{N}$, thereby generalizing well-known results of similar applications over the integers dating back to Gauss and lesser known results over the Gaussian integers \cite{Cross1983}. As another application, we prove that the Euler-Fermat Theorem extends naturally to the Eisenstein integers. This was done recently over the Gaussian integers but remained to be proven over the Eisenstein integers~\cite{Roberson2016}.

\begin{acknowledgments}
The authors acknowledge Edray Goins of Pomona College for helpful comments and suggestions. We especially thank Stan Wagon of Macalester College for sharing his Eisenstein ring \texttt{Mathematica} code with us; this code was an invaluable aid for our computations. We also acknowledge the faculty and students at Florida Gulf Coast University for their hospitality where we first presented these results. Finally, we acknowledge The Mousetrap in Eau Claire, Wisconsin for their hospitality where we conducted much of this research.
\end{acknowledgments}

\bigskip


\section{Preliminaries and definitions}

The \textit{Eisenstein integers}, denoted $\zrho$, is a subring of $\mathbb{C}$ defined as follows: $$\zrho = \{ a+b\rho \mid a,b \in \mathbb{Z} \text{ and } \rho = e^{2\pi i/3 }\}.$$ Notice that $\rho$ is the primitive third root of unity satisfying $x^3-1=0$. In particular, the minimal polynomial of $\rho$ is the quadratic $x^2+x+1$, or the third cyclotomic polynomial, and hence $\zrho$ like the Gaussian integers $\mathbb{Z}[i]$ is a \textit{quadratic integer ring}. These two quadratic integer rings tessellate the complex plane; in particular, the Eisenstein integers form a regular triangular (and consequently hexagonal) lattice in the complex plane as illustrated in the Figure~\ref{fig:eisenPlot}.

By abuse of language, for each $a + b\rho \in \zrho$ we denote $a$ as the \textit{real part} and $b$ as the \textit{rho part}. So each horizontal edge in Figure~\ref{fig:eisenPlot} represents a change of one unit in the real part, each diagonal NW-SE edge represents a change of one unit in the rho part, and each diagonal SW-NE edge represents a change of one unit in both parts as in Figure~\ref{fig:movingInLattice}.

\newpage

\begin{figure}[H]
\centering
\begin{tikzpicture}
\draw[cyan] (3,3.46410) -- (-1.5,-4.33012);
\draw[cyan] (3.5,2.59807) -- (-.5,-4.33012);
\draw[cyan] (4,1.73205) -- (.5,-4.33012);
\draw[cyan] (4.5,0.86602) -- (1.5,-4.33012);
\draw[cyan] (5,0) -- (2.5,-4.33012);
\draw[cyan] (1.5,4.33012) -- (-3,-3.46410);
\draw[cyan] (0.5,4.33012) -- (-3.5,-2.59807);
\draw[cyan] (-0.5,4.33012) -- (-4,-1.73205);
\draw[cyan] (-1.5,4.33012) -- (-4.5,-0.86602);
\draw[cyan] (-2.5,4.33012) -- (-5,0);
\draw[cyan] (3,-3.46410) -- (-1.5,4.33012);
\draw[cyan] (3.5,-2.59807) -- (-.5,4.33012);
\draw[cyan] (4,-1.73205) -- (.5,4.33012);
\draw[cyan] (4.5,-0.86602) -- (1.5,4.33012);
\draw[cyan] (5,0) -- (2.5,4.33012);
\draw[cyan] (1.5,-4.33012) -- (-3,3.46410);
\draw[cyan] (0.5,-4.33012) -- (-3.5,2.59807);
\draw[cyan] (-0.5,-4.33012) -- (-4,1.73205);
\draw[cyan] (-1.5,-4.33012) -- (-4.5,0.86602);
\draw[cyan] (-2.5,-4.33012) -- (-5,0);
\draw[cyan] (-2.5,4.33012) -- (2.5,4.33012);
\draw[cyan] (-3,3.46410) -- (3,3.46410);
\draw[cyan] (-3.5,2.59807) -- (3.5,2.59807);
\draw[cyan] (-4,1.73205) -- (4,1.73205);
\draw[cyan] (-4.5,0.86602) -- (4.5,0.86602);
\draw[cyan] (-4.5,-0.86602) -- (4.5,-0.86602);
\draw[cyan] (-4,-1.73205) -- (4,-1.73205);
\draw[cyan] (-3.5,-2.59807) -- (3.5,-2.59807);
\draw[cyan] (-3,-3.46410) -- (3,-3.46410);
\draw[cyan] (-2.5,-4.33012) -- (2.5,-4.33012);
\draw[green, thick, <->] (0,-5.5) -- (0,5.5);
\draw[red, very thick, <->] (-6,0) -- (6,0) node[above]{$\mathbb{R}$-axis}
node[pos=.079,below,black]{\footnotesize$\minus5$}
node[pos=.161,below,black]{\footnotesize$\minus4$}
node[pos=.245,below,black]{\footnotesize$\minus3$}
node[pos=.329,below,black]{\footnotesize$\minus2$}
node[pos=.413,below,black]{\footnotesize$\minus1$}
node[pos=.5835,below,black]{\footnotesize$1$}
node[pos=.667,below,black]{\footnotesize$2$}
node[pos=.75,below,black]{\footnotesize$3$}
node[pos=.832,below,black]{\footnotesize$4$}
node[pos=.918,below,black]{\footnotesize$5$};
\foreach \x in {-5,-4,...,5} {
\fill (\x,0) circle (2pt); }
\draw[blue, very thick, <->] (3,5.19615) -- (-3,-5.19615) node[left]{$\rho^2$-axis}
node[pos=.072,left,black]{\footnotesize$5\plus5\rho$}
node[pos=.155,left,black]{\footnotesize$4\plus4\rho$}
node[pos=.239,left,black]{\footnotesize$3\plus3\rho$}
node[pos=.322,left,black]{\footnotesize$2\plus2\rho$}
node[pos=.405,left,black]{\footnotesize$1\plus\rho$}
node[pos=.57,left,black]{\footnotesize$\minus1\minus\rho$}
node[pos=.655,left,black]{\footnotesize$\minus2\minus2\rho$}
node[pos=.737,left,black]{\footnotesize$\minus3\minus3\rho$}
node[pos=.821,left,black]{\footnotesize$\minus4\minus4\rho$}
node[pos=.904,left,black]{\footnotesize$\minus5\minus5\rho$};
\foreach \x in {-2.5,-2,...,2.5} {
\fill (\x,1.732050808*\x) circle (2pt); }
\draw[magenta, very thick, <->] (3,-5.19615) -- (-3,5.19615) node[left]{$\rho$-axis} 
node[pos=.072,left,black]{\footnotesize$\minus5\rho$}
node[pos=.155,left,black]{\footnotesize$\minus4\rho$}
node[pos=.239,left,black]{\footnotesize$\minus3\rho$}
node[pos=.322,left,black]{\footnotesize$\minus2\rho$}
node[pos=.405,left,black]{\footnotesize$\minus\rho$}
node[pos=.57,left,black]{\footnotesize$\rho$}
node[pos=.655,left,black]{\footnotesize$2\rho$}
node[pos=.739,left,black]{\footnotesize$3\rho$}
node[pos=.822,left,black]{\footnotesize$4\rho$}
node[pos=.906,left,black]{\footnotesize$5\rho$};
\foreach \x in {-2.5,-2,...,2.5} {
\fill (\x,-1.732050808*\x) circle (2pt); }
\end{tikzpicture}
\caption{The Eisenstein integers form a triangular lattice which tessellates the complex plane.}
\label{fig:eisenPlot}
\end{figure}
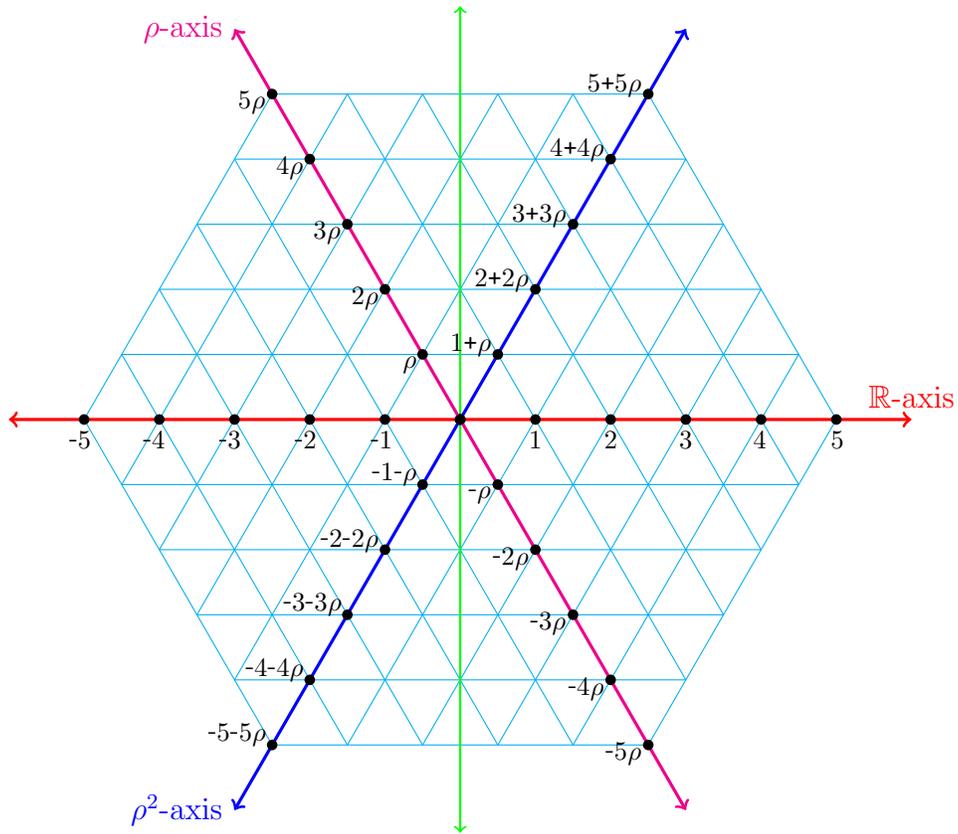

\vfill

\begin{figure}[H]
\centering
\begin{tikzpicture}[scale=.8]
\draw[cyan,thick] (-1,1.73205) -- (1,1.73205)
node[pos=-.1,above=-3pt,black]{\footnotesize$a\plus (b\plus 1)\rho$}
node[pos=1.2,above=-3pt,black]{\footnotesize$(a\plus 1)\plus (b\plus 1)\rho$};
\draw[cyan,thick] (-2,0) -- (2,0)
node[pos=.66,above=-3pt,black]{\footnotesize$a\plus b\rho$};
\draw[cyan,thick] (-1,-1.73205) -- (1,-1.73205)
node[pos=-.2,below=-1pt,black]{\footnotesize$(a\minus 1)\plus (b\minus 1)\rho$}
node[pos=1.1,below=-1pt,black]{\footnotesize$a\plus (b\minus 1)\rho$};
\draw[cyan,thick] (-1,-1.73205) -- (1,1.73205);
\draw[cyan,thick] (1,-1.73205) -- (-1,1.73205);
\draw[cyan,thick] (2,0) -- (1,1.73205)
node[pos=.11,right=-1pt,black]{\footnotesize$(a\plus 1)\plus b\rho$};
\draw[cyan,thick] (-1,1.73205) -- (-2,0)
node[pos=.89,left=-1pt,black]{\footnotesize$(a\minus 1)\plus b\rho$};
\draw[cyan,thick] (-2,0) -- (-1,-1.73205);
\draw[cyan,thick] (1,-1.73205) -- (2,0);
\fill (-1,1.73205) circle (2pt);
\fill (1,1.73205) circle (2pt);
\fill (-2,0) circle (2pt);
\fill (0,0) circle (2pt);
\fill (2,0) circle (2pt);
\fill (-1,-1.73205) circle (2pt);
\fill (1,-1.73205) circle (2pt);
\end{tikzpicture}
\caption{Starting at the point $a+b\rho$, a shift by one of the six units results in one of the six points on the hexagon surrounding $a+b\rho$.}
\label{fig:movingInLattice}
\end{figure}
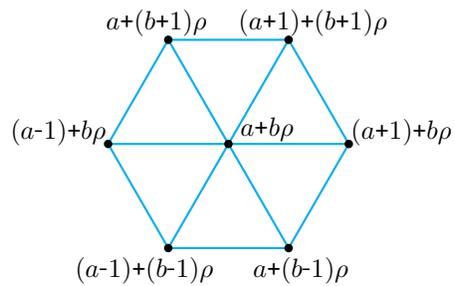

\vfill

\newpage

The following definitions and lemma are well known and can be found in multiple sources~\cite{Hardy1980,Ireland1990,Cox2013}.

\begin{definition}
\leavevmode 
\begin{itemize}
\item If $\theta = a+b\rho \in \zrho$, then the \textit{complex conjugate} of $\theta$, denoted $\overline{\theta}$, is defined $$\overline{\theta} = a + b\rho^2 = (a-b) - b\rho.$$
\item If $\theta = a+b\rho \in \zrho$, then the \textit{norm} of $\theta$ in given by $N(\theta) = \theta\overline{\theta} = a^2 -ab + b^2.$ The norm lies in $\mathbb{Z}$ and is multiplicative.
\item The element $\theta \in \zrho$ is a \textit{unit} if $\theta\eta = 1$ for some $\eta \in \zrho$.
\item The elements $\theta,\eta \in \zrho$ are \textit{associates} if $\theta = \delta\eta$ for some unit $\delta \in \zrho$. This is denoted $\theta \sim \eta$.
\end{itemize}
\end{definition}

\begin{lemma}\label{lem:units}[Ireland and Rosen~\cite[-
Proposition~9.1.1]{Ireland1990}]
If $\theta \in \zrho$, then $\theta$ is a unit if and only if $N(\theta) = 1$. The units in $\zrho$ are $1, -1, \rho, -\rho, \rho^2, \text{and} -\rho^2$.
\end{lemma}

\begin{remark}
The units in $\zrho$ form a cyclic group in the following sense:
\begin{align*}
\langle -\rho^2 \rangle &= \{1, -\rho^2, \rho, -1, \rho^2, -\rho \}\\
&= \{1, 1+\rho, \rho, -1, -1-\rho, -\rho \} = \left(\zrho\right)^\times \cong \mathbb{Z}_6.
\end{align*}
Figure~\ref{fig:normsPlot} gives these six units on the circle of norm 1. 
\end{remark}


\subsection{Even and ``odd'' Eisenstein integers}

In the integers the quotient by the ideal generated by the unique even prime 2, namely $\bigslant{\mathbb{Z}}{(2)}$, has two cosets whose elements partition $\mathbb{Z}$ into the even and odd integers. In the Eisenstein integers, of the six associates of norm 3, we distinguish the value $\beta = 1-\rho$ to play the role of the \textit{even} prime in $\zrho$. Unlike the $\mathbb{Z}$-setting, the quotient by the ideal generated by the even prime $\beta$, namely $\zrhomod{\beta}$, has three cosets whose elements partition $\zrho$ into three sets, which we call $\even$, $\odd_1$, and $\odd_2$. Theorem~\ref{thm:evenOdd} gives a simple characterization of when an Eisenstein integer is in one of these three sets. But first, we explicitly define the sets $\even$, $\odd_1$, and $\odd_2$.

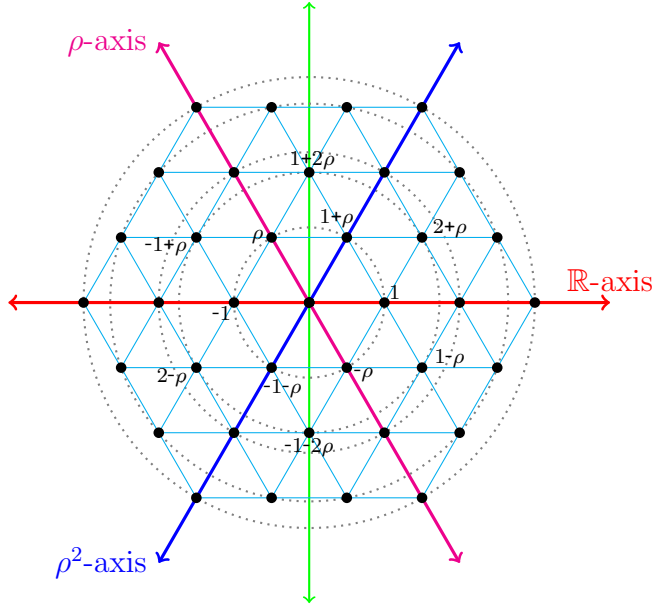
\begin{figure}
\centering
\begin{tikzpicture}
\draw[green, thick, <->] (0,-4) -- (0,4);
\draw[gray, thick, dotted] (0,0) circle (1cm);
\draw[gray, thick, dotted] (0,0) circle (1.73205cm);
\draw[gray, thick, dotted] (0,0) circle (2cm);
\draw[gray, thick, dotted] (0,0) circle (2.64575cm);
\draw[gray, thick, dotted] (0,0) circle (3cm);
\draw[cyan,thin] (2,1.73205) -- (-.5,-2.59807);
\draw[cyan,thin] (2.5,.86602) -- (.5,-2.59807);
\draw[cyan,thin] (3,0) -- (1.5,-2.59807);
\draw[cyan,thin] (.5,2.59807) -- (-2,-1.73205);
\draw[cyan,thin] (-.5,2.59807) -- (-2.5,-.86602);
\draw[cyan,thin] (-1.5,2.59807) -- (-3,0);
\draw[cyan,thin] (2,-1.73205) -- (-.5,2.59807);
\draw[cyan,thin] (2.5,-.86602) -- (.5,2.59807);
\draw[cyan,thin] (3,0) -- (1.5,2.59807);
\draw[cyan,thin] (.5,-2.59807) -- (-2,1.73205);
\draw[cyan,thin] (-.5,-2.59807) -- (-2.5,.86602);
\draw[cyan,thin] (-1.5,-2.59807) -- (-3,0);
\draw[cyan] (-1.5,2.59807) -- (1.5,2.59807);
\draw[cyan] (-2,1.73205) -- (2,1.73205)
		node[pos=.51,above=-3pt,black]{\scriptsize$1\plus2\rho$};
\draw[cyan] (-2.5,0.86602) -- (2.5,0.86602)
		node[pos=.119,below=-4pt,black]{\scriptsize$\minus1\plus\rho$}
		node[pos=.874,above=-4pt,black]{\scriptsize$2\plus\rho$};
\draw[cyan] (-2.5,-0.86602) -- (2.5,-0.86602)
		node[pos=.135,below=-4pt,black]{\scriptsize$2\minus\rho$}
		node[pos=.873,above=-4pt,black]{\scriptsize$1\minus\rho$};
\draw[cyan] (-2,-1.73205) -- (2,-1.73205)
		node[pos=.5,below=-2pt,black]{\scriptsize$\minus1\minus2\rho$};
\draw[cyan] (-1.5,-2.59807) -- (1.5,-2.59807);
\draw[red, very thick, <->] (-4,0) -- (4,0) node[above]{$\mathbb{R}$-axis}
node[pos=.355,below=-3pt,black]{\scriptsize$\minus1$}
node[pos=.642,above=-3pt,black]{\scriptsize$1$};
\foreach \x in {-3,-2,...,3} {
\fill (\x,0) circle (2pt);}
\draw[blue, very thick, <->] (2,3.46410) -- (-2,-3.46410) node[left]{$\rho^2$-axis}
node[pos=.339,left=-3pt,black]{\scriptsize$1\plus\rho$}
node[pos=.664,right=-3pt,black]{\scriptsize$\minus1\minus\rho$};
\foreach \x in {-1.5,-1,...,1.5} {
\fill (\x,1.732050808*\x) circle (2pt);}
\draw[magenta, very thick, <->] (2,-3.46410) -- (-2,3.46410) node[left]{$\rho$-axis}
node[pos=.364,right=-3pt,black]{\scriptsize$\minus\rho$}
node[pos=.626,left=-1.5pt,black]{\scriptsize$\rho$};
\foreach \x in {-1.5,-1,...,1.5} {
\fill (\x,-1.732050808*\x) circle (2pt);}
\fill (0,1.73205) circle (2pt);
\fill (0,-1.73205) circle (2pt);
\fill (1.5,.86605) circle (2pt);
\fill (-1.5,-.86605) circle (2pt);
\fill (1.5,-.86605) circle (2pt);
\fill (-1.5,.86605) circle (2pt);
\fill (2.5,.86605) circle (2pt);
\fill (2,1.73205) circle (2pt);
\fill (.5,2.59807) circle (2pt);
\fill (-.5,2.59807) circle (2pt);
\fill (-2,1.73205) circle (2pt);
\fill (-2.5,.86605) circle (2pt);
\fill (-2.5,-.86605) circle (2pt);
\fill (-2,-1.73205) circle (2pt);
\fill (-.5,-2.59807) circle (2pt);
\fill (.5,-2.59807) circle (2pt);
\fill (2,-1.73205) circle (2pt);
\fill (2.5,-.86605) circle (2pt);
\end{tikzpicture}
\caption{The Eisenstein integers with norm less than or equal to $9$.}
\label{fig:normsPlot}
\end{figure}

\begin{definition}\label{def:evenOdd}
The sets $\even$, $\odd_1$, and $\odd_2$ are defined as follows:
\begin{align*}
\even &= \{ a+b\rho \mid a+b\rho \equiv 0 \!\!\! \pmod{\beta}\}\\
\odd_1 &= \{ a+b\rho \mid a+b\rho \equiv 1 \!\!\! \pmod{\beta}\}\\
\odd_2 &= \{ a+b\rho \mid a+b\rho \equiv 2 \!\!\! \pmod{\beta}\}.
\end{align*}
\end{definition}

The following theorem gives a necessary and sufficient condition on the ``parity'' of an arbitrary $a+b\rho \in \zrho$ depending only on the sum $a+b$.

\begin{theorem}\label{thm:evenOdd}
Let $a+b\rho \in \zrho$. Then
\begin{align}
a+b\rho \in \even &\Longleftrightarrow a+b \equiv 0 \!\!\!\pmod{3}\label{thm:evenOdd_equivEven}\\
a+b\rho \in \odd_1 &\Longleftrightarrow a+b \equiv 1 \!\!\!\pmod{3}\label{thm:evenOdd_equivOdd1}\\
a+b\rho \in \odd_2 &\Longleftrightarrow a+b \equiv 2 \!\!\!\pmod{3}.\label{thm:evenOdd_equivOdd2}
\end{align}
\end{theorem}
\begin{proof}
Consider the following sequence of equivalences, where $a+b\rho \in \zrho$ and $c \in \mathbb{Z}$.
\begin{align*}
a+b\rho \equiv c \!\!\!\pmod{\beta} &\Longleftrightarrow \frac{a+b\rho-c}{\beta} \in \zrho\\
&\Longleftrightarrow \frac{a+b\rho-c}{1-\rho}\cdot\frac{1-\rho^2}{1-\rho^2} \in \zrho\\
&\Longleftrightarrow \frac{(a+b\rho-c)(2+\rho)}{3} \in \zrho\\
&\Longleftrightarrow \frac{2a+2b\rho-2c+a\rho-b-b\rho-c\rho}{3} \in \zrho\\
&\Longleftrightarrow \frac{2a-b-2c}{3} + \frac{a+b-c}{3}\rho \in \zrho\\
&\Longleftrightarrow \begin{cases} 2a-b \equiv 2c \!\!\!\pmod{3}\\ a+b \equiv c \!\!\!\pmod{3}.\end{cases}
\end{align*}
Observe that the two congruences are equivalent since $$2a-b \equiv 2c \!\!\!\!\pmod{3} \; \Longleftrightarrow \; 4a-2b \equiv 4c \!\!\!\!\pmod{3} \; \Longleftrightarrow \; a+b \equiv c \!\!\!\!\pmod{3}.$$ Thus $a+b\rho \equiv c \pmod{\beta}$ if and only if $a+b \equiv c \pmod{3}$. Substituting 0, 1, and 2 for $c$ proves the equivalences \eqref{thm:evenOdd_equivEven}, \eqref{thm:evenOdd_equivOdd1}, and \eqref{thm:evenOdd_equivOdd2}, respectively, hold as desired.
\end{proof}

Using the characterization given in Theorem~\ref{thm:evenOdd}, we can plot the distribution of these three sets in the triangular lattice, as illustrated in Figure~\ref{fig:evenOddPlot}.

It is clear that the product of an even integer and any other element in $\mathbb{Z}$ being even in $\mathbb{Z}$ extends to the product of an $\even$ integer and any other element in $\zrho$ being $\even$ in $\zrho$. Just as in $\mathbb{Z}$, where the product of two odd integers is always odd, the same property occurs in $\zrho$. However, in $\zrho$ the ``odds'' split into two equivalence classes. The following theorem classifies the parity of a product given the parity of the ``odd'' multiplicands.

\begin{theorem}\label{thm:evenOddMult}
Let $\tau \in \odd_1$ and $\sigma \in \odd_2$. Then 
\begin{align}
\tau \cdot \tau \in \odd_1 \label{thm:evenOddMult_11}\phantom{.}\\
\sigma \cdot \sigma \in \odd_1 \label{thm:evenOddMult_22}\phantom{.}\\
\tau \cdot \sigma \in \odd_2 \label{thm:evenOddMult_12}.
\end{align}
\end{theorem}
\begin{proof}
Let $a+b\rho$ and $c+d\rho$ in $\zrho$ such that $a+b \equiv m \pmod{3}$ and $c+d \equiv n \pmod{3}$. Notice that $(a+b\rho)\cdot(c+d\rho) = (ac-bd) + (ad+bc-bd)\rho$. It suffices to show that $ac+ad+bc-2bd \equiv mn \pmod{3}$.

Since $a+b \equiv m \pmod{3}$ and $c+d \equiv n \pmod{3}$, then $(a+b)\cdot(c+d) \equiv mn \pmod{3}$ implies $ac+bc+ad+bd \equiv mn \pmod{3}$. Then $ac+bc+ad+bd-3bd \equiv mn \pmod{3}$ implies $ac+bc+ad-2bd \equiv mn \pmod{3}$ as desired.

Applying Theorem~\ref{thm:evenOdd}, it follows that \eqref{thm:evenOddMult_11} holds if $m=n=1$, \eqref{thm:evenOddMult_22} holds if $m=n=2$, and \eqref{thm:evenOddMult_12} holds if $m=1$ and $n=2$.
\end{proof}

\begin{figure}[H]
\centering
\resizebox{.8\totalheight}{!}{
\begin{tikzpicture}
\draw[black,very thick,dashed] (0,0) circle (4cm);
\draw[cyan] (2.5,2.59807) -- (-1,-3.46410);
\draw[cyan] (3,1.73205) -- (0,-3.46410);
\draw[cyan] (3.5,0.86602) -- (1,-3.46410);
\draw[cyan] (4,0) -- (2,-3.46410);
\draw[cyan] (1,3.46410) -- (-2.5,-2.59807);
\draw[cyan] (0,3.46410) -- (-3,-1.73205);
\draw[cyan] (-1,3.46410) -- (-3.5,-0.86602);
\draw[cyan] (-2,3.46410) -- (-4,0);
\draw[cyan] (2.5,-2.59807) -- (-1,3.46410);
\draw[cyan] (3,-1.73205) -- (0,3.46410);
\draw[cyan] (3.5,-0.86602) -- (1,3.46410);
\draw[cyan] (4,0) -- (2,3.46410);
\draw[cyan] (1,-3.46410) -- (-2.5,2.59807);
\draw[cyan] (0,-3.46410) -- (-3,1.73205);
\draw[cyan] (-1,-3.46410) -- (-3.5,0.86602);
\draw[cyan] (-2,-3.46410) -- (-4,0);
\draw[cyan] (-2,3.46410) -- (2,3.46410);
\draw[cyan] (-2.5,2.59807) -- (2.5,2.59807);
\draw[cyan] (-3,1.73205) -- (3,1.73205);
\draw[cyan] (-3.5,0.86602) -- (3.5,0.86602);
\draw[cyan] (-3.5,-0.86602) -- (3.5,-0.86602);
\draw[cyan] (-3,-1.73205) -- (3,-1.73205);
\draw[cyan] (-2.5,-2.59807) -- (2.5,-2.59807);
\draw[cyan] (-2,-3.46410) -- (2,-3.46410);
\draw[green, thick, <->] (0,-5) -- (0,5);
\draw[red, very thick, <->] (-5,0) -- (5,0) node[above]{$\mathbb{R}$-axis}
	node[pos=.599,below,black]{$1$};
\draw[blue, very thick, <->] (2.5,4.33012) -- (-2.5,-4.33012) node[left]{$\rho^2$-axis};
\draw[magenta, very thick, <->] (2.5,-4.33012) -- (-2.5,4.33012) node[left]{$\rho$-axis}
	node[pos=.577,left=2pt,black]{$\rho$};
\foreach \x in {-1,0,1} {
\fill[orange] (-3,1.732050808*\x) circle (3.5pt);
\fill[orange] (3,1.732050808*\x) circle (3.5pt);}
\foreach \x in {-1.5,...,1.5} {
\fill[orange] (-1.5,1.732050808*\x) circle (3.5pt);
\fill[orange] (1.5,1.732050808*\x) circle (3.5pt);}
\foreach \x in {-2,...,2} {
\fill[orange] (0,1.732050808*\x) circle (3.5pt);}
\foreach \x in {-1.5,-.5,.5,1.5} {
\fill[purple] ([xshift=-3.5pt,yshift=-3.5pt]-.5,1.732050808*\x) rectangle ++(7pt,7pt);
\fill[purple] ([xshift=-3.5pt,yshift=-3.5pt]2.5,1.732050808*\x) rectangle ++(7pt,7pt);}
\foreach \x in {-2,...,2} {
\fill[purple] ([xshift=-3.5pt,yshift=-3.5pt]-2,1.732050808*\x) rectangle ++(7pt,7pt);
\fill[purple] ([xshift=-3.5pt,yshift=-3.5pt],1.732050808*\x) rectangle ++(7pt,7pt);}
\foreach \x in {-.5,.5} {
\fill[purple] ([xshift=-3.5pt,yshift=-3.5pt]-3.5,1.732050808*\x) rectangle ++(7pt,7pt);}
\foreach \x in {0} {
\fill[purple] ([xshift=-3.5pt,yshift=-3.5pt]4,1.732050808*\x) rectangle ++(7pt,7pt);}
\foreach \x in {-1.5,...,1.5} {
\fill[olive] (-2.65,1.732050808*\x-.129903) -- (-2.35,1.732050808*\x-.129903) -- (-2.5,1.732050808*\x+.129903) -- cycle;
\fill[olive] (.35,1.732050808*\x-.129903) -- (.65,1.732050808*\x-.129903) -- (.5,1.732050808*\x+.129903) -- cycle;}
\foreach \x in {-2,...,2} {
\fill[olive] (-1.15,1.732050808*\x-.129903) -- (-.85,1.732050808*\x-.129903) -- (-1,1.732050808*\x+.129903) -- cycle;
\fill[olive] (1.85,1.732050808*\x-.129903) -- (2.15,1.732050808*\x-.129903) -- (2,1.732050808*\x+.129903) -- cycle;}
\foreach \x in {-.5,.5} {
\fill[olive] (3.35,1.732050808*\x-.129903) -- (3.65,1.732050808*\x-.129903) -- (3.5,1.732050808*\x+.129903) -- cycle;}
\foreach \x in {0} {
\fill[olive] (-3.85,1.732050808*\x-.129903) -- (-4.15,1.732050808*\x-.129903) -- (-4,1.732050808*\x+.129903) -- cycle;}
\end{tikzpicture}}
\caption{The distribution of $\even$, $\odd_1$, and $\odd_2$  integers with norm less than 16. Orange circles represent $\even$, purple squares represent $\odd_1$, and green triangles represent $\odd_2$ integers.}
\label{fig:evenOddPlot}
\end{figure}
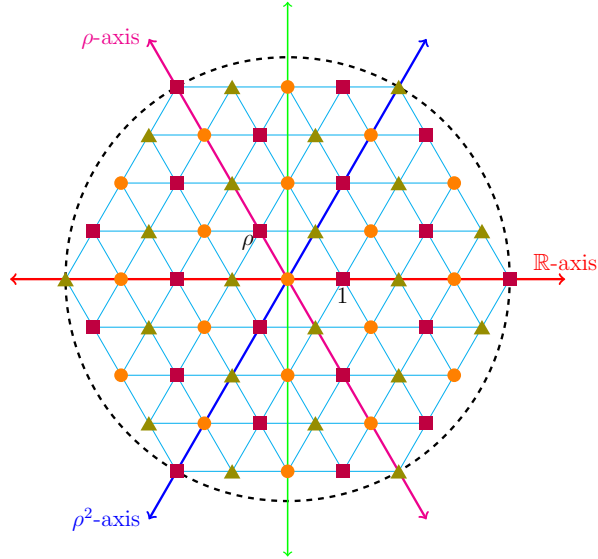

Whereas Theorem~\ref{thm:evenOdd} gives a necessary and sufficient condition for the parity of an Eisenstein integer $a+b\rho$ based solely on the residue class of the sum $a+b$ modulo $3$, the following theorem gives necessary (but not necessarily sufficient) conditions of the parity of $a+b\rho$ in terms of its norm.

\begin{theorem}\label{thm:parityNorm}
Let $a+b\rho \in \zrho$. Then
\begin{align}
a+b\rho \in \even &\Longleftrightarrow N(a+b\rho) \equiv 0 \!\!\!\pmod{3}\label{thm:parityNorm_equivEven}\\
a+b\rho \in \odd_1 &\Longrightarrow N(a+b\rho) \equiv 1 \!\!\!\pmod{3}\label{thm:parityNorm_equivOdd1}\\
a+b\rho \in \odd_2 &\Longrightarrow N(a+b\rho) \equiv 1 \!\!\!\pmod{3}.\label{thm:parityNorm_equivOdd2}
\end{align}
\end{theorem}
\begin{proof}
Consider $a+b\rho \in \zrho$ where $a+b \equiv n \pmod{3}$. Then 
\begin{align*}
(a+b)(a+b) \equiv n\cdot n \!\!\!\!\pmod{3} &\Longleftrightarrow a^2+2ab+b^2 \equiv n^2 \!\!\!\!\pmod{3}\\
&\Longleftrightarrow a^2+2ab+b^2-3ab \equiv n^2 \!\!\!\!\pmod{3}\\
&\Longleftrightarrow a^2-ab+b^2 \equiv n^2 \!\!\!\!\pmod{3}\\
&\Longleftrightarrow N(a+b\rho) \equiv n^2 \!\!\!\!\pmod{3}.
\end{align*}

Applying Theorem~\ref{thm:evenOdd}, it follows that \eqref{thm:parityNorm_equivEven} holds if $n=0$, \eqref{thm:parityNorm_equivOdd1} holds if $n=1$, and \eqref{thm:parityNorm_equivOdd2} holds if $n=2$.
\end{proof}

\begin{remark}
Since Theorem~\ref{thm:parityNorm} determines the remainder when the norm of any Eisenstein integer $a+b\rho$ is divided by $3$, we conclude that $N(a+b\rho)$ is never congruent to $2$ modulo $3$.
\end{remark}


\subsection{The three types of Eisenstein primes}

This subsection outlines the three types of primes in $\zrho$ which are well known and can be found in many sources~\cite{Hardy1980,Ireland1990}.

\begin{proposition}\label{prop:primes}[Hardy and Wright~\cite[
Theorem~255]{Hardy1980}]
Every prime element $\gamma$ in $\zrho$ falls into one of three categories (up to associates):
\begin{enumerate}
\item the even prime $\gamma = 1-\rho$,
\item the rational primes $\gamma = p$ where $p$ is prime in $\mathbb{Z}$ and $p \equiv 2 \pmod{3}$, and
\item the nonassociate primes $\gamma$ and $\overline{\gamma}$ where $q$ is prime in $\mathbb{Z}$, the norm $N(\gamma)$ equals $q$, and $q \equiv 1 \pmod{3}$.
\end{enumerate}
\end{proposition}

In Figure~\ref{fig:primePlot}, the three types of Eisenstein primes are plotted on the lattice. The even prime and its five associates lie on the solid orange circle. The rational primes $\gamma=2,5$ and their five associates lie on the two dashed purple circles. The primes $\gamma$ such that $N(\gamma)=7,13,19$ lie on the remaining three dotted olive circles.

\begin{figure}
\centering
\resizebox{.8\totalheight}{!}{
\begin{tikzpicture}
\draw[cyan] (3,3.46410) -- (-1.5,-4.33012);
\draw[cyan] (3.5,2.59807) -- (-.5,-4.33012);
\draw[cyan] (4,1.73205) -- (.5,-4.33012);
\draw[cyan] (4.5,0.86602) -- (1.5,-4.33012);
\draw[cyan] (5,0) -- (2.5,-4.33012);
\draw[cyan] (1.5,4.33012) -- (-3,-3.46410);
\draw[cyan] (0.5,4.33012) -- (-3.5,-2.59807);
\draw[cyan] (-0.5,4.33012) -- (-4,-1.73205);
\draw[cyan] (-1.5,4.33012) -- (-4.5,-0.86602);
\draw[cyan] (-2.5,4.33012) -- (-5,0);
\draw[cyan] (3,-3.46410) -- (-1.5,4.33012);
\draw[cyan] (3.5,-2.59807) -- (-.5,4.33012);
\draw[cyan] (4,-1.73205) -- (.5,4.33012);
\draw[cyan] (4.5,-0.86602) -- (1.5,4.33012);
\draw[cyan] (5,0) -- (2.5,4.33012);
\draw[cyan] (1.5,-4.33012) -- (-3,3.46410);
\draw[cyan] (0.5,-4.33012) -- (-3.5,2.59807);
\draw[cyan] (-0.5,-4.33012) -- (-4,1.73205);
\draw[cyan] (-1.5,-4.33012) -- (-4.5,0.86602);
\draw[cyan] (-2.5,-4.33012) -- (-5,0);
\draw[cyan] (-2.5,4.33012) -- (2.5,4.33012);
\draw[cyan] (-3,3.46410) -- (3,3.46410);
\draw[cyan] (-3.5,2.59807) -- (3.5,2.59807);
\draw[cyan] (-4,1.73205) -- (4,1.73205);
\draw[cyan] (-4.5,0.86602) -- (4.5,0.86602);
\draw[cyan] (-4.5,-0.86602) -- (4.5,-0.86602);
\draw[cyan] (-4,-1.73205) -- (4,-1.73205);
\draw[cyan] (-3.5,-2.59807) -- (3.5,-2.59807);
\draw[cyan] (-3,-3.46410) -- (3,-3.46410);
\draw[cyan] (-2.5,-4.33012) -- (2.5,-4.33012);
\draw[green, thick, <->] (0,-5.5) -- (0,5.5);
\draw[red, very thick, <->] (-6,0) -- (6,0) node[above]{$\mathbb{R}$-axis};
\draw[blue, very thick, <->] (3,5.19615) -- (-3,-5.19615) node[left]{$\rho^2$-axis};
\draw[magenta, very thick, <->] (3,-5.19615) -- (-3,5.19615) node[left]{$\rho$-axis};
\foreach \x in {1,...,6} {
\fill[orange] (\x*360/6 +30: 1.73205cm) circle (3.5pt); 
\fill[purple] (\x*360/6: 2cm) circle (3.5pt); 
\fill[purple] (\x*360/6: 5cm) circle (3.5pt); 
\fill[olive] (\x*360/6 +19.10660: 2.64575cm) circle (3.5pt); 
\fill[olive] (\x*360/6 +40.89339: 2.64575cm) circle (3.5pt); 
\fill[olive] (\x*360/6 +13.89788: 3.60555cm) circle (3.5pt); 
\fill[olive] (\x*360/6 +46.10211: 3.60555cm) circle (3.5pt); 
\fill[olive] (\x*360/6 +23.41322: 4.35889cm) circle (3.5pt); 
\fill[olive] (\x*360/6 +36.58677: 4.35889cm) circle (3.5pt); 
}
\draw[orange,thick] (0,0) circle (1.73205cm); 
\draw[purple,thick,dashed] (0,0) circle (2cm); 
\draw[purple,thick,dashed] (0,0) circle (5cm); 
\draw[olive,very thick,dotted] (0,0) circle (2.64575cm); 
\draw[olive,very thick,dotted] (0,0) circle (3.60555cm); 
\draw[olive,very thick,dotted] (0,0) circle (4.35889cm); 
\end{tikzpicture}}
\caption{All the Eisenstein primes with norms less than or equal to 19.}
\label{fig:primePlot}
\end{figure}
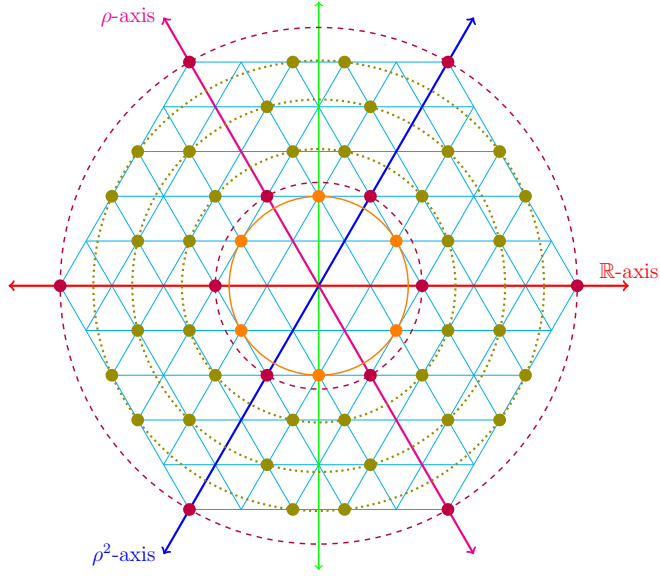

\bigskip

\boitegrise{
\begin{convention}\label{convention_on_notation}
\vspace{-1.8ex}
From this point forward, we let $\beta$ denote the prime $1-\rho$ and let $p$ and $q$ be primes in $\mathbb{Z}$ such that $p \equiv 2 \pmod{3}$ and $q \equiv 1 \pmod{3}$. Let $q = \psi\overline{\psi}$ where $\psi$ and $\overline{\psi}$ are non-associate prime factors of $q$ in $\zrho$. We also denote a generic prime in $\zrho$ by $\gamma$. \vspace{-.3in}
\end{convention}}{0.9\textwidth}
\vspace{1em}

We now establish the associates of $\beta^n$ for $n \in \mathbb{Z}$.

\begin{proposition}\label{prop:betasim}
One particular associate of $\beta^{2m}$ and $\beta^{2m+1}$, respectively, is given by $3^m$ and $3^m \beta$, respectively. That is, we have $\beta^{2m} \sim 3^m$ and $\beta^{2m+1} \sim 3^m\beta$.
\end{proposition}

\begin{proof}
Observe that
\begin{align*}
\beta^{2m} = (\beta^2)^m &= ((1-\rho) (1-\rho))^m\\
&= (1 - 2\rho + \rho^2)^m\\
&= (1 - 2\rho + (-1-\rho))^m\\
&= (-3\rho)^m\\
&\sim (-\rho^2)^m(-3\rho)^m\\
&= (3\rho^3)^m\\
&= 3^m
\end{align*}
Hence $\beta^{2m} \sim 3^m$. As a consequence, we have $\beta^{2m+1} \sim 3^m\beta$.
\end{proof}


\subsection{Unique factorization in \texorpdfstring{$\zrho$}{Z-rho}}

As noted in the introduction, the Eisenstein integer ring is a unique factorization domain. Unlike its analog, the integers, where a natural number clearly factors uniquely into products of powers of primes, factoring an Eisenstein integer is not as straightforward as the following example displays.

\begin{example}\label{exam:eta_as_product_of_powers_of_primes}
Consider $\eta = 48 - 72 \rho \in \zrho$. We will factor this uniquely (up to associates) as a product of powers of primes. We start by factoring out all the powers of $\beta$. Notice that $48 - 72 = -24 \equiv 0 \pmod{3}$, so $\eta$ is an $\even$ integer and we can factor out at least one instance of $\beta$. Since $\frac{48-72\rho}{\beta} = 56-8\rho$, we have $\eta = \beta \cdot (56-8\rho)$. Notice that our multiplicand $56-8\rho$ is $\even$ as well, so we factor out another instance of $\beta$ resulting in $\eta = \beta^2 \cdot (40+16\rho)$. Now our new multiplicand is not $\even$ since $40+16=56 \equiv 2 \pmod{3}$, so we have factored out all possible powers of $\beta$.

Next we factor out Category 2 primes from $40+16\rho$. Notice that $2^3$ divides both $40$ and $16$, yielding $\eta = \beta^2 \cdot 2^3 \cdot (5+2\rho)$ since $2$ is a Category 2 prime.

Lastly, we factor $5+2\rho$ into Category 3 primes. Notice that $N(5+2\rho) = 19$. Since $19$ is a rational prime and $19 \equiv 1 \pmod{3}$, we conclude that $5+2\rho$ is itself a Category 3 prime. Thus the unique factorization of $\eta$ is $$\eta = \beta^2 \cdot 2^3 \cdot (5+2\rho).$$
\end{example}


\section{Classes of \texorpdfstring{$\zrho/(\gamma^n)$}{Z-rho mod gamma} for a prime \texorpdfstring{$\gamma$}{gamma} in \texorpdfstring{$\zrho$}{Z-rho}}

Computing the Euler phi function on an integer $n \in \mathbb{Z}$, using the unit group definition of $\phi(n)$, is a relatively easy task as it is straightforward to find a nice set of coset representatives in $\bigslant{\mathbb{Z}}{(n)}$---for example, we can just take the $n$ least residue class members $0, 1, \ldots, n-1$ as our representatives. This is aided by the natural ordering possessed by the ring $\mathbb{Z}$. However, since this natural ordering does not exist in $\zrho$, it is not a straight-forward task to find a complete set of residue classes for $\zrhomod{\eta}$ where $\eta$ is an arbitrary Eisenstein integer. To that end, we first find a complete set of residue classes for $\zrhomod{\gamma^n}$ where $\gamma$ is prime in $\zrho$.

\subsection{Equivalence classes of \texorpdfstring{$\zrho/(\gamma^n)$}{Z-rho mod gamma}}

Before we find the complete set of residue classes for each quotient $\zrhomod{\gamma^n}$ where $\gamma$ is prime in $\zrho$, we prove the following lemma.

\begin{lemma}\label{lem:intDivideEisen}
If $n \in \mathbb{Z}$ and $a+b\rho \in \zrho$, then $n \divides a+b\rho$ if and only if $n \divides a$ and $n \divides b$. 
\end{lemma}
\begin{proof}
Since $n \divides a+b\rho$, we know $n(c+d\rho) = a+b\rho$ for some $c+d\rho \in \zrho$. Then $cn=a$ and $dn=b$. Hence $n \divides a$ and $n \divides b$. The converse holds similarly.
\end{proof}

\begin{theorem}\label{thm:equivClasses}
The complete sets of residue classes of $\zrhomod{\gamma^n}$ are given by:
\begin{align}
&\zrhomod{\beta^{2m}} = \{ [a+b\rho] \mid 0 \leq a,b \leq 3^m-1 \}, \label{thm:equivClasses_1}\\
&\zrhomod{\beta^{2m+1}} = \{ [a+b\rho] \mid 0 \leq a \leq 3^{m+1}-1 \text{\, and \, } 0 \leq b \leq 3^m-1 \}, \label{thm:equivClasses_2}\\
&\zrhomod{p^n} = \{ [a+b\rho] \mid 0 \leq a,b \leq p^n-1 \}, \label{thm:equivClasses_3}\\
&\zrhomod{\psi^n} = \{ [a] \mid 0 \leq a \leq q^n-1 \}. \label{thm:equivClasses_4}
\end{align}
\end{theorem}

\begin{proof}
For each set of equivalence classes, we show they are distinct and nonrepeating representatives. There are four cases:

\noindent\textbf{CASE 1:} It suffices to show that the $3^{2m}$ classes $$\{ [a+b\rho] \mid 0 \leq a,b \leq 3^m-1 \}$$ of $\zrhomod{\beta^{2m}}$ are distinct and nonrepeating. By Proposition~\ref{prop:betasim}, we know $\beta^{2
m}$ and $3^m$ are associates, so it follows that $\zrhomod{\beta^{2m}} = \zrhomod{3^m}$. Consider two representatives $a+b\rho$ and $c+d\rho$ of the same class such that $0 \leq a,b,c,d \leq 3^m-1$. Thus $[a+b\rho] = [c+d\rho]$ implies $a+b\rho \equiv c+d\rho \pmod{3^m}$. Then $3^m \divides (a-c) + (b-d)\rho$ implies $3^m \divides a-c$ and $3^m \divides b-d$ by Lemma~\ref{lem:intDivideEisen}. Since $0 \leq a,b,c,d \leq 3^m-1$ we have $|a-c| \leq 3^m-1$, and thus $3^m \divides a-c$ forces $a=c$. By a similar argument $b=d$ follows. Therefore the $3^{2m}$ classes in \eqref{thm:equivClasses_1} are distinct.

\noindent\textbf{CASE 2:} It suffices to show that the $3^{2m+1}$ classes $$\{ [a+b\rho] \mid 0 \leq a \leq 3^{m+1}-1 \text{\, and \, } 0 \leq b \leq 3^m-1 \}$$ of $\zrhomod{\beta^{2m+1}}$ are distinct and nonrepeating.  By Proposition~\ref{prop:betasim}, we know $\beta^{2m+1}$ and $3^m\beta$ are associates, so it follows that $\zrhomod{\beta^{2m+1}} = \zrhomod{3^m\beta}$. Consider two representatives $a+b\rho$ and $c+d\rho$ of the same class such that $0 \leq a,c \leq 3^{m+1}-1$ and $0 \leq b,d \leq 3^m-1$. Thus $[a+b\rho] = [c+d\rho]$ implies $a+b\rho \equiv c+d\rho \pmod{\beta^{2m+1}}$. Then $\beta^{2m+1} \divides (a-c) + (b-d)\rho$. Since $3^m\beta$ is an associate of $\beta^{2m+1}$, then $3^m\beta \divides (a-c) + (b-d)\rho$. In particular $3^m \divides b-d$ by Lemma~\ref{lem:intDivideEisen}. Since $0 \leq b,d \leq 3^m-1$ we have $|b-d| \leq 3^m-1$, and thus $3^m \divides b-d$ forces $b=d$. 

Now since $b=d$ and $3^m\beta \divides (a-c) + (b-d)\rho$, we know that $3^m\beta \divides a-c$, which implies $3^m \divides a-c$. Recall that $0 \leq a,c \leq 3^{m+1}-1$. Hence $3^m \divides a-c$ implies that either $|a-c| = 3^m$ or $|a-c| = 0$. If $|a-c| = 3^m$, then $3^m\beta \divides 3^m$ implies $\beta \divides 1$. Thus $\beta$ is a unit, which is a contradiction. Hence $|a-c| = 0$, which forces $a=c$. Therefore the $3^{2m+1}$ classes in \eqref{thm:equivClasses_2} are distinct.

\noindent\textbf{CASE 3:} It suffices to show that the $p^{2n}$ classes $$\{ [a+b\rho] \mid 0 \leq a,b \leq p^n-1 \}$$ of $\zrhomod{p^n}$ are distinct and nonrepeating. Consider two representatives $a+b\rho$ and $c+d\rho$ of the same class such that $0 \leq a,b,c,d \leq p^n-1$. Thus $[a+b\rho] = [c+d\rho]$ implies $a+b\rho \equiv c+d\rho \pmod{p^n}$. Then $p^n \divides (a-c) + (b-d)\rho$ implies $p^n \divides a-c$ and $p^n \divides b-d$ by Lemma~\ref{lem:intDivideEisen}. Since $0 \leq a,b,c,d \leq p^n-1$ we have $|a-c| \leq p^n-1$, and thus $p^n \divides a-c$ forces $a=c$. By a similar argument $b=d$ follows. Therefore the $p^{2n}$ classes in \eqref{thm:equivClasses_3} are distinct.

\noindent\textbf{CASE 4:} It suffices to show that the $q^n$ classes $$\{ [a] \mid 0 \leq a \leq q^n-1 \}$$ of $\zrhomod{\psi^n}$ are distinct and nonrepeating. Consider two representatives $a$ and $b$ of the same class such that $0 \leq a,b \leq q^n-1$. Thus $[a] = [b]$ implies $a \equiv b \pmod{\psi^n}$. Then $\psi^n \divides a-b$. Let $\psi^n\delta = a-b$ for some $\delta \in \zrho$. Then $\overline{\psi^n\delta} = \overline{a-b} = a-b$ implies $\overline{\psi}^n\overline{\delta} = a-b$ and thus $\overline{\psi}^n \divides a-b$. Since $\psi$ and $\overline{\psi}$ are nonassociate primes and both $\psi^n$ and $\overline{\psi}^n$ divide $a-b$, then $\psi^n\overline{\psi}^n \divides a-b$. Thus $q^n \divides a-b$. Since $0 \leq a,b \leq q^n-1$ we have $|a-b| \leq q^n-1$, and thus $q^n \divides a-b$ forces $a=b$. Therefore the $q^n$ classes in \eqref{thm:equivClasses_4} are distinct.

We now show that an arbitrary element $\theta = a+b\rho$ in $\zrho$ belongs to exactly one of the classes in each of the four sets.

\noindent\textbf{CASE $\mathbf{1^\prime}$:} Consider $x,y \in \mathbb{Z}$ such that $x \equiv a \pmod{3^m}$ and $y \equiv b \pmod{3^m}$ where $0 \leq x,y \leq 3^m-1$. Then $\theta = a+b\rho \equiv x+y\rho \pmod{3^m}$ implies $\theta \in [x+y\rho]$ in \eqref{thm:equivClasses_1}.

\noindent\textbf{CASE $\mathbf{2^\prime}$:} Consider $x,y \in \mathbb{Z}$ such that $x \equiv a \pmod{3^{m+1}}$ and $y \equiv b \pmod{3^{m+1}}$ where $0 \leq x,y \leq 3^{m+1}-1$. This implies that $\theta = a+b\rho \equiv x+y\rho \pmod{3^{m+1}}$. If $y < 3^m$, then $\theta \in [x+y\rho]$ in $\zrhomod{\beta^{2m+1}}$. On the other hand, if $y \geq 3^m$, then
\begin{align*}
x+y\rho &= x+y\rho +3^m\beta -3^m\beta\\
&= x+y\rho + 3^m(1-\rho) -3^m\beta\\
&= x+y\rho + 3^m - 3^m\rho -3^m\beta\\
&= (x+3^m) + (y-3^m)\rho -3^m\beta.
\end{align*}
So $\theta \equiv (x+3^m) + (y-3^m)\rho \pmod{3^m\beta}$ where $0 \leq y-3^m \leq 3^m-1$. Observe that $x+3^m$ can be reduced by multiples of $3^{m+1}$ yielding $x' \equiv x+3^m \pmod{3^{m+1}}$ where $0 \leq x' \leq 3^{m+1}-1$. Therefore it follows that $\theta \in [x' + (y-3^m)\rho]$ in \eqref{thm:equivClasses_2}.

\noindent\textbf{CASE $\mathbf{3^\prime}$:} Consider $x,y \in \mathbb{Z}$ such that $x \equiv a \pmod{p^n}$ and $y \equiv b \pmod{p^n}$ where \phantom{zzzzz} $0 \leq x,y \leq p^n-1$. Then $\theta = a+b\rho \equiv x+y\rho \pmod{p^n}$ implies $\theta \in [x+y\rho]$ in \eqref{thm:equivClasses_3}.

\noindent\textbf{CASE $\mathbf{4^\prime}$:} It suffices to show that $a,b$ and $\rho$ each belong to one of the classes in $$\{ [c] \mid 0 \leq c \leq q^n-1\}.$$ Clearly we can reduce $a$ and $b$ by multiples of $q^n$ so that $a \in [c_1]$ and $b \in [c_2]$ where $0 \leq c_1,c_2 \leq q^n-1$. Now we show $\rho$ belongs to one of the desired classes. Let $\psi^n = x-y\rho$ so that $y\rho \equiv x \pmod{\psi^n}$. Notice that $q \ndivides y$, for if $q \divides y$, then $\psi\overline{\psi} \divides y$ implies $\psi \divides y$. Since $\psi$ is prime, then $y\rho \equiv x \pmod{\psi^n}$ implies $y\rho \equiv x \pmod{\psi}$. Thus $\psi \divides x$. Then $\overline{\psi} \divides x$, so $\psi\overline{\psi} \divides x$ since $\psi$ and $\overline{\psi}$ are distinct primes. Thus $q \divides x$. Since $q \divides x$ and $q \divides y$, we get $q \divides x-y\rho$ implies $q \divides \psi^n$. But $q = \psi\overline{\psi}$ yields a contradiction. Hence $q \ndivides y$ so $q$ and $y$ are relatively prime, which implies $q^n$ and $y$ are relatively prime. Thus the congruence $zy \equiv 1 \pmod{q^n}$ is solvable in $\mathbb{Z}$. 
Recall $y\rho \equiv x \pmod{\psi^n}$. Then $zy\rho \equiv zx \pmod{\psi^n}$. Since $zy \equiv 1 \pmod{q^n}$, then $zy \equiv 1 \pmod{\psi^n\overline{\psi}^n}$ implies $zy - 1 = \psi^n\overline{\psi}^n\eta$ for some $\eta \in \zrho$. So $zy - 1 = \psi^n(\overline{\psi}^n\eta)$ implies $zy \equiv 1 \pmod{\psi^n}$. Then $zy\rho \equiv \rho \pmod{\psi^n}$. We have $zy\rho \equiv zx \pmod{\psi^n}$ and $zy\rho \equiv \rho \pmod{\psi^n}$, yielding $\rho \equiv zx \pmod{\psi^n}$. Since $\psi^n \divides q^n$, then $q^n$ is a multiple of $\psi^n$. Thus $zx$ can be reduced by multiples of $q^n$, yielding $c_3 \equiv zx \pmod{q^n}$ where $0 \leq c_3 \leq q^n-1$. Hence $\rho \in [c_3]$ in $\zrhomod{\psi^n}$. Since $a,b$ and $\rho$ each belong to the classes $[c_1], [c_2],$ and $[c_3]$, respectively, then let $C \equiv c_1 + c_2c_3 \pmod{q^n}$ where $0 \leq C \leq q^n-1$. We conclude that $\theta \in [C]$ in \eqref{thm:equivClasses_4}.
\end{proof}


\subsection{Criteria for when two classes in \texorpdfstring{$\zrho/(\gamma^n)$}{Z-rho mod gamma} are equivalent}

In the following four theorems, we give ways to find equivalent equivalence classes for each of the quotients $\zrhomod{\gamma^n}$ where $\gamma$ is prime in $\zrho$. Furthermore, these theorems allow us to show when two Eisenstein integers represent the same equivalence class.

\begin{remark}
There is no analog of these four theorems in Cross' work on the Gaussian integers~\cite{Cross1983}, but they are implicitly implied in some examples he gives. In particular, our Theorem~\ref{thm:classesInBetaOdd} shows that the equivalence of two classes in $\bigslant{\zrho}{(\beta^{2m+1})}$ is more restrictive than one might expect.
\end{remark}

\begin{theorem}\label{thm:classesInBetaEven}
If $[a+b\rho] \in \zrhomod{\beta^{2m}}$, then $[a+b\rho] = [(a+3^mk) + (b+3^mj)\rho]$ for all $k,j \in \mathbb{Z}$. 
\end{theorem}
\begin{proof}
Let $a+b\rho \in \zrho$ and $k,j \in \mathbb{Z}$. Then
\begin{align*}
(a+3^mk) + (b+3^mj)\rho &= (a+b\rho) + (3^mk+3^mj\rho)\\
&= (a+b\rho) + 3^m(k+j\rho).
\end{align*}
So $(a+3^mk) + (b+3^mj)\rho \equiv a+b\rho \pmod{3^m}$. By Lemma~\ref{prop:betasim}, we know $\beta^{2m} \sim 3^m$. Hence $(a+3^mk) + (b+3^mj)\rho \equiv a+b\rho \pmod{\beta^{2m}}$, and the result follows.
\end{proof}

\begin{theorem}\label{thm:classesInBetaOdd}
If $[a+b\rho] \in \bigslant{\zrho}{(\beta^{2m+1})}$, then $[a+b\rho] = [(a+3^mk) + (b+3^mj)\rho]$ for $k,j \in \mathbb{Z}$ such that $k+j \equiv 0 \pmod{3}$. 
\end{theorem}
\begin{proof}
Let $a+b\rho \in \zrho$ and $k,j \in \mathbb{Z}$. Observe that
\begin{align*}
[a+b\rho] = [(a+3^mk) + (b+3^mj)\rho] &\Longleftrightarrow a+b\rho \equiv (a+3^mk) + (b+3^mj)\rho \!\!\!\pmod{\beta^{2m+1}}\\
&\Longleftrightarrow \frac{((a+3^mk) + (b+3^mj)\rho) - (a+b\rho)}{\beta^{2m+1}} \in \zrho.
\end{align*}
So it suffices to show when $$\frac{((a+3^mk) + (b+3^mj)\rho) - (a+b\rho)}{\beta^{2m+1}}\in\zrho.$$ Observe that
$$\frac{((a+3^mk) + (b+3^mj)\rho) - (a+b\rho)}{\beta^{2m+1}} = \frac{3^mk + 3^mj\rho}{3^m\beta} = \frac{k+j\rho}{\beta}.$$
Thus $k+j\rho \in \even$, so $k+j \equiv 0 \pmod{3}$ as desired.
\end{proof}

\begin{theorem}\label{thm:classesInP}
If $[a+b\rho] \in \zrhomod{p^n}$, then $[a+b\rho] = [(a+p^nk) + (b+p^nj)\rho]$ for all $k,j \in \mathbb{Z}$. 
\end{theorem}
\begin{proof}
Let $a+b\rho \in \zrho$ and $k,j \in \mathbb{Z}$. Then
\begin{align*}
(a+p^nk) + (b+p^nj)\rho &= (a+b\rho) + (p^nk+p^nj\rho)\\
&= (a+b\rho) + p^n(k+j\rho).
\end{align*}
Hence $(a+p^nk) + (b+p^nj)\rho \equiv a+b\rho \pmod{p^n}$, and the result follows.
\end{proof}

\begin{theorem}\label{thm:classesInPsi}
For $\psi^n = x+y\rho$ and $\psi^n\overline{\psi}^n = q^n$, if $[a+b\rho] \in \zrhomod{\psi^n}$ and $c \in \mathbb{Z}$ then $[a+b\rho] = [c]$ when the congruence $-ay+bx \equiv -cy \pmod{q^n}$ holds.
\end{theorem}
\begin{proof}
Let $a+b\rho \in \zrho$ and $c \in \mathbb{Z}$. Consider $\psi^n = x+y\rho$ where $\psi^n\overline{\psi^n} = q^n$. Observe that
\begin{align*}
[a+b\rho] = [c] &\Longleftrightarrow a+b\rho \equiv c \!\!\!\pmod{\psi^n}\\
&\Longleftrightarrow \frac{(a+b\rho) - c}{\psi^n} \in \zrho.
\end{align*}
So it suffices to show when $\frac{(a+b\rho) - c}{\psi^n}$ is in $\zrho$. Observe that
\begin{align*}
\frac{(a+b\rho) - c}{\psi^n} &= \frac{a+b\rho - c}{\psi^n}\cdot\frac{\overline{\psi}^n}{\overline{\psi}^n}\\
&= \frac{(a+b\rho-c)(x+y\rho^2)}{q^n}\\
&= \frac{ax + bx\rho - cx - ay - ay\rho + by + cy + cy\rho}{q^n}\\
&= \frac{ax-cx-ay+by+cy}{q^n} + \frac{bx-ay+cy}{q^n}\rho
\end{align*}
Thus $\frac{(a+b\rho) - c}{\psi^n}$ is in $\zrho$ when the following congruences hold:
\begin{center}
$\begin{cases} 
a(x-y)+by \equiv c(x-y) &\pmod{q^n}\\ 
-ay+bx \equiv -cy &\pmod{q^n}.
\end{cases}$
\end{center}
We now show that the two congruences are equivalent. Starting with the top congruence, observe that
\begin{align}
&a(x-y)+by \equiv c(x-y) \!\!\!\pmod{q^n}\nonumber\\ 
&\hspace{.5in}\Longrightarrow (-y)(x-y)^{-1}\left(a(x-y)+by\right) \equiv (-y)(x-y)^{-1}c(x-y) \!\!\!\pmod{q^n}\nonumber\\
&\hspace{.5in}\Longrightarrow -ay + (-y^2)(x-y)^{-1}(b) \equiv -cy \!\!\!\pmod{q^n}\nonumber\\
&\hspace{.5in}\Longrightarrow -ay+bx \equiv -cy \!\!\!\pmod{q^n}\label{line:lem_classesInPsi}.
\end{align}
Implication~\eqref{line:lem_classesInPsi} holds since 
\begin{align*}
q^n = x^2+y^2-xy &\Longrightarrow 0 \equiv x^2+y^2-xy \!\!\!\pmod{q^n}\\
&\Longrightarrow -y^2 \equiv x(x-y) \!\!\!\pmod{q^n}\\
&\Longrightarrow (x-y)^{-1}(-y^2) \equiv x \!\!\!\pmod{q^n}.
\end{align*}
Hence the two congruences are equivalent, so $\frac{(a+b\rho) - c}{\psi^n} \in \zrho$ when $-ay+bx \equiv -cy \pmod{q^n}$ and the result follows. 
\end{proof}

\begin{example}
Consider $\psi = 3+\rho$ and $q=7$ and let $[17-3\rho] \in \zrhomod{\psi^2}$. Then $\psi^2 = 8+5\rho$ and $q^2=49$. We want to find $[c] \in \zrhomod{\psi^2}$ where $c \in \mathbb{Z}$ such that $[17-3\rho] = [c]$. By Theorem~\ref{thm:classesInPsi}, we need to solve the congruence $$(-17) \cdot 5 + (-3) \cdot 8 \equiv -c \cdot 5 \pmod{49}.$$ Observe that
\begin{align*}
(-17) \cdot 5 + (-3) \cdot 8 \equiv -c \cdot 5 \pmod{49} &\Longrightarrow -109 \equiv -5c \pmod{49}\\
&\Longrightarrow 11 \equiv 5c \pmod{49}\\
&\Longrightarrow 110 \equiv 50c \pmod{49}\\
&\Longrightarrow 110 \equiv c \pmod{49}\\
&\Longrightarrow 12 \equiv c \pmod{49}.
\end{align*}
Hence $[17-3\rho] = [12] \in \zrhomod{8+5\rho}$. One can confirm that this is the case by showing that $17-3\rho$ and $12$ differ by a multiple of $8+5\rho$, specifically $-\rho$.
\end{example}


\subsection{Units in \texorpdfstring{$\zrho/(\gamma^n)$}{Z-rho mod gamma}}

We now find the units in each of the quotient rings $\bigslant{\zrho}{(\gamma^n)}$ where $\gamma$ is prime in $\zrho$.

\begin{theorem}\label{thm:unitsInRings}
Let $a, b \in \mathbb{Z}$. Then the units in $\zrhomod{\gamma^n}$ are given by
\begin{align}
&[a+b\rho] \text{ where } a + b \not\equiv 0 \!\!\!\pmod{3} &&\text{ if } \gamma=\beta \label{thm:unitsInRings_1}\\
&[a+b\rho] \text{ where } \gcd{(a,p)}=1 \text{ or } \gcd{(b,p)}=1 &&\text{ if } \gamma=p \label{thm:unitsInRings_2}\\
&[a] \text{ where } \gcd{(a,q)}=1 &&\text{ if } \gamma=\psi \label{thm:unitsInRings_3}
\end{align}
\end{theorem}

\begin{proof}
Let $\gamma$ be prime in $\zrho$. Let $\theta \in \zrho$ and consider the class $[\theta]$ in $\zrhomod{\gamma^n}$. Then 
\begin{align*}
[\theta] \text{ is a unit in } \zrhomod{\gamma^n} &\Longleftrightarrow [\theta][\delta] = [1] \text{ for some } [\delta] \in \zrhomod{\gamma^n}\\
&\Longleftrightarrow \theta\delta \equiv 1 \pmod{\gamma^n}\\
&\Longleftrightarrow \theta\delta + \gamma^n\eta = 1 \text{ for some } \eta \in \zrho\\
&\Longleftrightarrow \theta\delta + \gamma(\gamma^{n-1}\eta) = 1\\
&\Longleftrightarrow \gcd{(\theta, \gamma)} = 1.
\end{align*}
\noindent\textbf{CASE 1:} $(\gamma=\beta)$ Let $[a+b\rho]$ in $\zrhomod{\beta^n}$. Then $[a+b\rho]$ is a unit if and only if $\gcd{(a+b\rho,\beta)} = 1$, that is, if and only if $\beta$ does not divide $a+b\rho$. Note the latter implies that $a+b\rho \notin \even$. By the contrapositive of equivalence \eqref{thm:evenOdd_equivEven} of Theorem~\ref{thm:evenOdd}, it follows that \eqref{thm:unitsInRings_1} holds.

\noindent\textbf{CASE 2:} $(\gamma=p)$ Let $[a+b\rho]$ in $\zrhomod{p^n}$. Then $[a+b\rho]$ is a unit if and only if $\gcd{(a+b\rho,p)} = 1$, that is, if and only if $p$ does not divide $a+b\rho$. By Lemma~\ref{lem:intDivideEisen}, we have $p \ndivides a+b\rho$ if and only if $p \ndivides a$ or $p \ndivides b$, or equivalently, $\gcd{(a,p)} = 1$ or $\gcd{(b,p)} = 1$. Hence \eqref{thm:unitsInRings_2} holds.

\noindent\textbf{CASE 3:} $(\gamma=\psi)$ Let $[a]$ in $\zrhomod{\psi^n}$. Then $[a]$ is a unit if and only if $\gcd{(a,\psi)} = 1$, that is, if and only if $\psi$ does not divide $a+b\rho$. But $\psi \ndivides a$ implies $\psi\overline{\psi} \ndivides a$, or equivalently $q \ndivides a$. Hence $[a]$ is a unit if and only if $\gcd{(a,q)} = 1$, so \eqref{thm:unitsInRings_3} holds.
\end{proof}

\begin{example}
Consider $\zrhomod{\beta^2}$. By Theorem~\ref{thm:equivClasses}, we have $$\zrhomod{\beta^2} = \{ [a+b\rho] \mid 0 \leq a,b \leq 2 \} = \{ [0], [1], [2], [\rho], [1+\rho], [2+\rho], [2\rho], [1+2\rho], [2+2\rho] \}.$$ By Theorem~\ref{thm:unitsInRings}, the units in $\zrhomod{\beta^2}$ are given by $$\zrhomodunit{\beta^2}= \{ [1], [2], [\rho], [1+\rho], [2\rho], [2+2\rho] \}.$$ To find the inverse of $[2+2\rho]$ in $\zrhomod{\beta^2}$ it suffices to solve $(2+2\rho)x \equiv 1 \pmod{\beta^2}$. This statement is equivalent to $(2+2\rho)x + \beta^2(-k) = 1$ for some $k \in \zrho$. Then using $\beta^2 = -3\rho$ and the Euclidean algorithm, we have
\begin{align*}
2+2\rho &= \rho(-3\rho) + (-1-\rho)\\
\rho &= (-1-\rho)(-1-\rho) + 0.
\end{align*}
So $-1-\rho = (2+2\rho) - \rho(-3\rho)$ implies $-1-\rho = (2+2\rho) + \beta^2(-\rho)$. Multiplying both sides by $\rho$ yields $1 = (2+2\rho)\rho + \beta^2(-\rho^2)$. Hence $x=\rho$, so $[\rho]$ is the multiplicative inverse of $[2+2\rho]$ in $\zrhomod{\beta^2}$.
\end{example}


\section{An Euler phi function \texorpdfstring{$\ephi$}{phi-sub-rho} for the Eisenstein integers}\label{sec:Eulerphi}

We are now ready to compute the values of the Euler phi function on the Eisenstein integers. Analogous to the integer setting we have the following definition.

\begin{definition}\label{def:ephi}
Let $\eta \in \zrho-\{0\}$. We define the \textit{Euler phi function} $\ephi$ on $\eta$ to be the cardinality of $\zrhomodunit{\eta}$.
\end{definition}


\subsection{Euler phi function on powers of primes in \texorpdfstring{$\zrho$}{Z-rho}}\label{subsec:Eulerphi_powersOfPrimes}

In order to introduce the Euler phi function $\ephi$ for the Eisenstein integers in this section, we need to first compute the number of units in $\zrhomod{\gamma^n}$ for an arbitrary prime $\gamma \in \zrho$ and $n \in \mathbb{N}$. The following lemma is helpful in proving the subsequent theorem.

\begin{lemma}\label{lem:pchunks}
Let $p$ be prime in $\mathbb{Z}$ and $n\geq1$. Then there are exactly $p^{n-1}$ nonnegative integers less than $p^n$ that are congruent to $k$ modulo $p$ for each $0 \leq k \leq p-1$.
\end{lemma}
\begin{proof}
Let $X_k = \{a \in [0,p^{n}-1] \mid a \equiv k \pmod{p} \}$ for some $k \in [0,p-1]$. The set $X_k$ can be rewritten as $\{ a \mid a = k+pj \text{ for } 0 \leq j \leq p^{n-1}-1 \}$ which clearly has size $p^{n-1}$.
\end{proof}

The following main theorem gives the Euler phi function on powers of primes in the $\zrho$. It may be helpful to recall the three distinct categories of primes from Proposition~\ref{prop:primes} and the notation in Convention~\ref{convention_on_notation}.

\begin{theorem}\label{thm:sizeOfUnits}
For each prime $\gamma \in \zrho$, the cardinality $\Big|\zrhomodunit{\gamma^n}\Big|$ is given by
\begin{align}
\ephi{(\gamma^n)} &= 3^{n}-3^{n-1} &\hspace{-1in}\text{ if } \gamma=\beta \label{thm:sizeOfUnits_1}\\
\ephi{(\gamma^n)} &= p^{2n}-p^{2n-2} &\hspace{-1in}\text{ if } \gamma=p \label{thm:sizeOfUnits_2}\\
\ephi{(\gamma^n)} &= q^{n}-q^{n-1} &\hspace{-1in}\text{ if } \gamma=\psi \label{thm:sizeOfUnits_3}.
\end{align}
\end{theorem}
\begin{proof}
To find the cardinality of the group of units $\zrhomodunit{\gamma^n}$, it suffices to find the number of nonunits in $\zrhomod{\gamma^n}$ and subtract this from the cardinality of $\zrhomod{\gamma^n}$ as given in Theorem~\ref{thm:equivClasses}.

\noindent\textbf{CASE 1:} $(\gamma=\beta)$
\begin{enumerate}[a.]
\item $(n = 2m)$ By Theorem~\ref{thm:equivClasses}, the size of $\zrhomod{\beta^{2m}} = \{ [a+b\rho] \mid 0 \leq a,b \leq 3^m-1 \}$ is $3^{2m}$. By Theorem~\ref{thm:unitsInRings}, the set of nonunits in $\zrhomod{\beta^{2m}}$ is given by $$X_{\beta^{2m}} = \{ [a+b\rho] \in \zrhomod{\beta^{2m}} \mid a \equiv -b \!\!\! \pmod{3} \}.$$ By Lemma~\ref{lem:pchunks}, there are $3^{m-1}$ and $3^{m}$ choices for the values of $a$ and $b$, respectively. Thus we have $\big| X_{\beta^{2m}} \big| = 3^{m-1}3^{m} = 3^{2m-1}.$ Hence $\Big|\zrhomodunit{\beta^{2m}}\Big| = 3^{2m}-3^{2m-1}$.
\item $(n = 2m+1)$ By Theorem~\ref{thm:equivClasses}, the size of $$\zrhomod{\beta^{2m+1}} = \{ [a+b\rho] \mid 0 \leq a \leq 3^{m+1}-1 \text{\, and \, } 0 \leq b \leq 3^m-1 \}$$ is $3^{2m+1}$. By Theorem~\ref{thm:unitsInRings}, the set of nonunits in $\zrhomod{\beta^{2m+1}}$ is given by $$X_{\beta^{2m+1}} = \{ [a+b\rho] \in \zrhomod{\beta^{2m+1}} \mid a \equiv -b \!\!\! \pmod{3}\}.$$ By Lemma~\ref{lem:pchunks}, there are $3^{m}$ choices for each of the values $a$ and $b$. Thus we have $\big| X_{\beta^{2m+1}} \big| = 3^{m}3^{m} = 3^{2m}$. Hence $\Big|\zrhomodunit{\beta^{2m+1}}\Big| = 3^{2m+1}-3^{2m}$.
\end{enumerate}
In both the even and the odd case,  $\Big|\zrhomodunit{\beta^{n}}\Big| = 3^{n}-3^{n-1}$, so \eqref{thm:sizeOfUnits_1} holds.

\noindent\textbf{CASE 2:} $(\gamma=p)$ By Theorem~\ref{thm:equivClasses}, the size of $\zrhomod{p^n} = \{ [a+b\rho] \mid 0 \leq a,b \leq p^n-1 \}$ is $p^{2n}$. By Theorem~\ref{thm:unitsInRings}, the set of nonunits in $\zrhomod{p^n}$ is given by $$X_{p^n} = \{ [a+b\rho] \in \zrhomod{p^n} \mid \gcd{(a,p)}\not= 1 \text{ and } \gcd{(b,p)}\not= 1\}.$$ The set $X_{p^n}$ can be rewritten as $\{ [a+b\rho] \in \zrhomod{p^n} \mid p \text{ divides } a \text{ and } b \}$. By Lemma~\ref{lem:pchunks}, there are $p^{n-1}$ choices for each of the values $a$ and $b$. Thus $\big| X_{p^n} \big| = p^{n-1}p^{n-1} = p^{2n-2}$. Hence $\Big|\zrhomodunit{p^{n}}\Big| = p^{2n}-p^{2n-2}$, so \eqref{thm:sizeOfUnits_2} holds.

\noindent\textbf{CASE 3:} $(\gamma=\psi)$ By Theorem~\ref{thm:equivClasses}, the size of $\zrhomod{\psi^n} = \{ [a] \mid 0 \leq a \leq q^n-1 \}$ is $q^{n}$. By Theorem~\ref{thm:unitsInRings}, the set of nonunits in $\zrhomod{\psi^n}$ is given by $$X_{\psi^n} = \{ [a] \in \zrhomod{\psi^n} \mid \gcd{(a,q)} \neq 1 \}.$$ The set $X_{\psi^n}$ can be rewritten as $\{ [a] \in \zrhomod{\psi^n} \mid q \text{ divides } a \}$. By Lemma~\ref{lem:pchunks}, there are $q^{n-1}$ choices for the value of $a$. Thus $\big| X_{\psi^n} \big| = q^{n-1}$. Hence $\Big|\zrhomodunit{\psi^{n}}\Big| = q^{n}-q^{n-1}$, so \eqref{thm:sizeOfUnits_3} holds.
\end{proof}

\begin{remark}[Value of $\ephi(1)$]\label{rem:ephi_on_1}
In the $\mathbb{Z}$-setting, any two integers are congruent modulo 1 and hence $\bigslant{\mathbb{Z}}{(1)}$ has only one congruence class. Thus $\left( \bigslant{\mathbb{Z}}{(1)} \right)^\times$ is the trivial group so $\varphi(1)$ is defined to be 1. Similarly in the $\zrho$-setting, we define $\ephi$ applied to 1 or any of its five associates to equal 1.
\end{remark}


\subsection{\texorpdfstring{$\ephi$}{Euler phi function on Eisenstein} is multiplicative}

Since $\zrho$ is a unique factorization domain, we know that any arbitrary $\eta \in \zrho$ can be written uniquely (up to associates) as a product of powers of primes $\eta = \gamma_1^{k_1} \ldots \gamma_r^{k_r}$. Hence if we can show that $\ephi$ is a multiplicative function, then we can compute $\ephi$ on any arbitrary element in $\zrho$. The following important theorem allows us to conclude that $\ephi$ indeed is multiplicative.

\begin{theorem}\label{thm:unitsMultiplicative}
Let $\eta,\theta \in \zrho$ such that $\gcd{(\eta,\theta)} = 1$. Let $$f: \zrhomodunit{\eta} \times \zrhomodunit{\theta} \rightarrow \; \zrhomod{\eta\theta}$$ such that for $[\sigma] \in \zrhomodunit{\eta}$ and $[\tau] \in \zrhomodunit{\theta}$, we have $f(([\sigma],[\tau])) = [\upsilon]$ where $$\upsilon \equiv \sigma \!\!\!\! \pmod{\eta} \text{\quad and \quad} \upsilon \equiv \tau \!\!\!\! \pmod{\theta}.$$ Then the image of $f$ is given by $\im{f} = \zrhomodunit{\eta\theta}$ and $$f: \zrhomodunit{\eta} \times \zrhomodunit{\theta} \rightarrow \zrhomodunit{\eta\theta}$$ is an isomorphism.
\end{theorem}

\begin{proof}
If $([\sigma_1],[\tau_1]) = ([\sigma_2],[\tau_2])$ for $[\sigma_1],[\sigma_2] \in \zrhomodunit{\eta}$ and $[\tau_1],[\tau_2] \in \zrhomodunit{\theta}$, then $f(([\sigma_1],[\tau_1])) = f(([\sigma_2],[\tau_2]))$ by the Chinese Remainder Theorem. So $f$ is well defined. 

Let $[\sigma] \in \zrhomodunit{\eta}$ and $[\tau] \in \zrhomodunit{\theta)}$ such that $f(([\sigma],[\tau])) = [\upsilon]$. Since $\upsilon \equiv \sigma \pmod{\eta}$ and $\gcd{(\sigma,\eta)}=1$, then $\gcd{(\upsilon,\eta)}=\gcd{(\sigma,\eta)}=1$. Similarly we have $\gcd{(\upsilon,\theta)}=\gcd{(\tau,\theta)}=1$. Thus $\gcd{(\upsilon,\eta\theta)}=1$ since $\gcd{(\upsilon,\eta)}=1$ and $\gcd{(\upsilon,\theta)}=1$. Hence $f(([\sigma],[\tau])) = [\upsilon] \in \zrhomodunit{\eta\theta}$, so $\im{f} \subseteq \zrhomodunit{\eta\theta}$.

Since $\im{f} \subseteq \zrhomodunit{\eta\theta}$, it suffices to show $\zrhomodunit{\eta\theta} \subseteq \im{f}$. To that end, let $[\chi] \in \zrhomodunit{\eta\theta}$. Then $\gcd{(\chi,\eta\theta)}=1$ implies $\gcd{(\chi,\eta)}=1$ and $\gcd{(\chi,\theta)}=1$ since $\gcd{(\eta,\theta)}=1$. So $\chi \equiv \sigma^\prime \pmod{\eta}$ for some $\sigma^\prime \in \zrhomodunit{\eta}$ and $\chi \equiv \tau^\prime \pmod{\theta}$ for some $\tau^\prime \in \zrhomodunit{\theta}$. Observe that $f(([\sigma^\prime],[\tau^\prime]))=[\chi]$ lies in $\im(f)$ so $\zrhomodunit{\eta\theta} \subseteq \im{f}$. Hence $\im{f} = \zrhomodunit{\eta\theta}$.


Now we show that $f: \zrhomodunit{\eta} \times \zrhomodunit{\theta} \rightarrow \zrhomodunit{\eta\theta}$ is an isomorphism. We just showed that $f$ surjects onto $\zrhomodunit{\eta\theta}$. To show $f$ is injective, it suffices to show that $\ker{f}$ is trivial. Observe that
\begin{align*}
\ker{f} &= \{([\sigma],[\tau]) \in \zrhomodunit{\eta} \times \zrhomodunit{\theta} \mid f(([\sigma],[\tau]))=[1] \in \zrhomodunit{\eta\theta}\}\\
&= \{([\sigma],[\tau]) \mid 1 \equiv \sigma \!\!\!\!\! \pmod{\eta} \text{ and } 1 \equiv \tau \!\!\!\!\! \pmod{\theta}\}\\
&= \{([1],[1])\}.
\end{align*}
Hence $f: \zrhomodunit{\eta} \times \zrhomodunit{\theta} \rightarrow \zrhomodunit{\eta\theta}$ is a bijection. It is trivial to show $f$ is a ring homomorphism and thus $f$ is an isomorphism.
\end{proof}

\begin{corollary}\label{cor:phi_is_mult}
The Euler phi function $\ephi$ on the Eisenstein integers is multiplicative.
\end{corollary}
\begin{proof}
Let $\eta,\theta \in \zrho$ such that $\gcd{(\eta,\theta)}=1$. By Theorem~\ref{thm:unitsMultiplicative}, we know $$\zrhomodunit{\eta\theta} \cong \zrhomodunit{\eta} \times \zrhomodunit{\theta}.$$ Since by Definition~\ref{def:ephi} the symbols $\ephi{(\eta\theta)}, \, \ephi{(\eta)}, \text{ and } \ephi{(\theta)}$ are defined to be the cardinalities of $\zrhomodunit{\eta\theta}, \, \zrhomodunit{\eta} \text{ and } \zrhomodunit{\theta}$, respectively, we conclude that $\ephi{(\eta\theta)} = \ephi{(\eta)}\ephi{(\theta)}$.
\end{proof}

\begin{example}
Recall that in Example~\ref{exam:eta_as_product_of_powers_of_primes}, we showed that $\eta = 48 - 72 \rho$ can be written as the following product of powers of primes:
$$\eta = 48 - 72 \rho = 2^3 \cdot \beta^2 \cdot (5 + 2 \rho).$$ We can now compute $\ephi{(\eta)}$ given its prime factorization: 
\begin{align*}
\ephi{(\eta)} &= \ephi{(2^3 \cdot \beta^2 \cdot (5 + 2 \rho))}\\
&= \ephi{(2^3)} \cdot \ephi{(\beta^2)} \cdot \ephi{(5 + 2 \rho)} & \mbox{by Corollary~\ref{cor:phi_is_mult}}\\
&= \Big|\zrhomodunit{2^3}\Big| \cdot \Big|\zrhomodunit{\beta^2}\Big| \cdot \Big|\zrhomodunit{5 + 2 \rho}\Big|\\
&= (2^6 - 2^4) \cdot (3^2 - 3^1) \cdot (19^1 - 19^0) & \mbox{by Theorem~\ref{thm:sizeOfUnits}}\\
&= 5184.
\end{align*}
\end{example}


\section{Applications}

\subsection{Group structures and a cyclicity sufficiency criterion for certain unit groups}

\begin{example}[The group structure of $\zrhomodunit{\beta^2}$]
By Theorem~\ref{thm:sizeOfUnits}, we know that $\zrhomodunit{\beta^2}$ has $6$ elements. Thus $\zrhomodunit{\beta^2}$ is isomorphic to either $\mathbb{Z}_6$ or $S_3$. We give the Cayley table for $\zrhomodunit{\beta^2}$ below.
$$
\begin{array}{|c||c|c|c|c|c|c|}
\hline
\equiv \!\!\!\!\!\! \pmod{\beta^2} & 1 & 1+\rho & \rho & 2 & 2+2\rho & 2\rho\\
\hline
\hline
1 & 1 & 1+\rho & \rho & 2 & 2+2\rho & 2\rho\\
\hline
1+\rho & 1+\rho & \rho & 2 & 2+2\rho & 2\rho & 1\\
\hline
\rho & \rho & 2 & 2+2\rho & 2\rho & 1 & 1+\rho\\
\hline
2 & 2 & 2+2\rho & 2\rho & 1 & 1+\rho & \rho\\
\hline
2+2\rho & 2+2\rho & 2\rho & 1 & 1+\rho & \rho & 2\\
\hline
2\rho & 2\rho & 1 & 1+\rho & \rho & 2 & 2+2\rho\\
\hline
\end{array}
$$
Notice that both $1+\rho$ and $2\rho$ have order $6$, so they are generators for the group. Hence $\zrhomodunit{\beta^2}$ is cyclic, so it must be isomorphic to $\mathbb{Z}_6$.
\end{example}

\begin{example}[The group structure of $\zrhomodunit{\beta^3}$]
By Theorem~\ref{thm:equivClasses}, we know that $$\bigslant{\zrho}{(\beta^3)} = \{ [a+b\rho] \mid 0 \leq a \leq 8 \text{ and } 0 \leq b \leq 2\}.$$
By Theorem~\ref{thm:sizeOfUnits}, we know that $\zrhomodunit{\beta^3}$ has $18$ elements. These $18$ elements are given by 
\begin{align*}
\zrhomodunit{\beta^3} &= \{ [a+b\rho] \in \bigslant{\zrho}{(\beta^3)} \mid a+b \not\equiv 0 \pmod{3} \}\\
&= \{ [1],[2],[4],[5],[7],[8],[\rho],[1+\rho],[3+\rho],[4+\rho],[6+\rho],[7+\rho],\\
&\qquad\![2\rho],[2+2\rho],[3+2\rho],[5+2\rho],[6+2\rho],[8+2\rho]\}
\end{align*}
according to Theorem~\ref{thm:unitsInRings}. Thus $\zrhomodunit{\beta^3}$ is isomorphic to one of the five groups of order $18$. However, since $\zrho$ is a commutative ring, every ideal in $\zrho$ is commutative. So $\zrhomod{\beta^3}$ is a commutative quotient ring, which implies that $\zrhomodunit{\beta^3}$ is an abelian group. There are only two abelian groups of order $18$: $\mathbb{Z}_{18}$ and $\mathbb{Z}_6 \times \mathbb{Z}_3$. It suffices to find the order of every element.

Let $a+b\rho \in \zrhomodunit{\beta^3}$ such that the order of $a+b\rho$ is $n$. So $(a+b\rho)^n \equiv 1 \pmod{\beta^3}$. Let $(a+b\rho)^n = c+d\rho$. Then $c+d\rho \equiv 1 \pmod{\beta^3}$ implies $\frac{(c+d\rho)-1}{\beta^3} \in \zrho$. Using $\beta^3 = -3-6\rho$, observe that
\begin{align*}
\frac{(c+d\rho)-1}{\beta^3} &= \frac{(c-1)+d\rho}{-3-6\rho}\cdot\frac{-3-6\rho^2}{-3-6\rho^2}\\
&= \frac{-3c+3-3d\rho-6c\rho^2+6\rho^2-6d\rho^2}{27}\\
&= \frac{-3c+3-3d\rho+6c+6c\rho-6-6\rho-6d}{27}\\
&= \frac{2c-2d-1}{9} + \frac{2c-d-2}{9}\rho.
\end{align*}
Thus $c+d\rho \equiv 1 \pmod{\beta^3}$ if $\frac{2c-2d-1}{9} + \frac{2c-d-2}{9}\rho \in \zrho$. We use this to find the orders of all $18$ elements, which are given below in red.
\begin{center}
	\begin{tabular}{|ll|ll|ll|ll|ll|ll|}
	\hline
  $[1]$ &\textcolor{red}{1} & $[2]$ &\textcolor{red}{6} & $[4]$ &\textcolor{red}{3} & $[5]$ &\textcolor{red}{6} & $[7]$ &\textcolor{red}{3} & $[8]$ &\textcolor{red}{2}\\
  \hline
	$[\rho]$ &\textcolor{red}{3} & $[1+\rho]$ &\textcolor{red}{6} & $[3+\rho]$ &\textcolor{red}{3} & $[4+\rho]$ &\textcolor{red}{6} & $[6+\rho]$ &\textcolor{red}{3} & $[7+\rho]$ &\textcolor{red}{6}\\
	\hline
  $[2\rho]$ &\textcolor{red}{6} & $[2+2\rho]$ &\textcolor{red}{6} & $[3+2\rho]$ &\textcolor{red}{6} & $[5+2\rho]$ &\textcolor{red}{3} & $[6+2\rho]$ &\textcolor{red}{6} & $[8+2\rho]$ &\textcolor{red}{3}\\
  \hline
	\end{tabular}
\end{center}
None of the elements have order $18$, so we conclude that $\zrhomodunit{\beta^3}$ is isomorphic to $\mathbb{Z}_6 \times \mathbb{Z}_3$.
\end{example}

\begin{theorem}\label{thm:structOfPsi}
Consider $\psi\overline{\psi} = q$ where $q \equiv 1 \pmod{3}$ where $q$ is prime in $\mathbb{Z}$ and $\psi$ and $\overline{\psi}$ are prime in $\zrho$. Then $\zrhomodunit{\psi^n} \cong \mathbb{Z}_{q^n-q^{n-1}}$.
\end{theorem}
\begin{proof}
It suffices to show that the map $f: \left(\bigslant{\mathbb{Z}}{(q^n)}\right)^\times \rightarrow \zrhomodunit{\psi^n}$ such that $f([a]) = [a]$ is an isomorphism. Let $[a],[b] \in \left(\bigslant{\mathbb{Z}}{(q^n)}\right)^\times$. Observe that $$f([a][b])=f([ab])=[ab]=[a][b]=f([a])f([b]),$$ so $f$ is a homomorphism. Now let $[c] \in \zrhomodunit{\psi^n}$. Trivially there exists $[c] \in \left(\bigslant{\mathbb{Z}}{(q^n)}\right)^\times$ such that $[c] \in \left(\bigslant{\mathbb{Z}}{(q^n)}\right)^\times \mapsto [c] \in \zrhomodunit{q^n}$ under $f$, so $f$ is surjective.

Lastly, we know that $\Big|\left(\bigslant{\mathbb{Z}}{(q^n)}\right)^\times\Big| = q^n-q^{n-1}$ by the Euler phi function on $\mathbb{Z}$. Also $\Big|\zrhomodunit{\psi^n}\Big| = q^n-q^{n-1}$ by Theorem~\ref{thm:sizeOfUnits}. Hence $\Big|\left(\bigslant{\mathbb{Z}}{(q^n)}\right)^\times\Big| = \Big|\zrhomodunit{\psi^n}\Big|$, so $f$ is an isomorphism and these two unit groups are algebraically equivalent. It is well known that $\left(\bigslant{\mathbb{Z}}{(q^n)}\right)^\times \cong \mathbb{Z}_{q^n-q^{n-1}}$ for a prime $q \in \mathbb{Z}$. Therefore $\zrhomodunit{\psi^n} \cong \mathbb{Z}_{q^n-q^{n-1}}$.
\end{proof}

\begin{example}
We find the structure of $\zrhomodunit{-6\rho}$ and a complete set of residue classes. Observe that $-6\rho = -3\rho \cdot 2 = \beta^2 \cdot 2$. As an application of Theorem~\ref{thm:unitsMultiplicative}, we find an explicit isomorphism $$f: \zrhomodunit{\beta^2} \times \zrhomodunit{2} \longmapsto \zrhomodunit{-6\rho}.$$ By Theorems~\ref{thm:equivClasses}~and~\ref{thm:unitsInRings}, we have
\begin{align*} 
\zrhomodunit{\beta^2} &= \begin{cases} [a+b\rho] \, {\biggm|} \!\!\!\!\!\begin{rcases} &0 \leq a,b \leq 2 \text{ and}\\ &a \not\equiv -b \!\!\!\pmod{3} \end{rcases}\end{cases}\\
&= \{ 1, 1+\rho, \rho, 2, 2\rho, 2+2\rho\}
\end{align*}
and
\begin{align*}
\zrhomodunit{2} &= \begin{cases} [a+b\rho] \, {\Biggm|} \!\!\!\!\!\begin{rcases} &0 \leq a,b \leq 1 \text{ and}\\ &\gcd{(a,2)}=1 \text{ or }\! \gcd{(b,2)}=1 \end{rcases}\end{cases}\\
&= \{ 1, \rho, 1+\rho\}.
\end{align*}
It suffices to show where each of the 18 elements $([\sigma],[\tau])$ in $\zrhomodunit{\beta^2} \times \zrhomodunit{2}$ map to in $\zrhomodunit{-6\rho}$. For example, to find the image of some $([\sigma],[\tau])$, we solve the system of congruences
$$
\begin{cases}
a+b\rho \equiv \sigma \!\!\!\pmod{\beta^2}\\
a+b\rho \equiv \tau \!\!\!\pmod{2}
\end{cases}
$$
via the Chinese Remainder Theorem (CRT). By the proof of CRT, it is not hard to see that a solution modulo $\beta^2\cdot2$ is
\begin{align*}
a+b\rho &= \sigma\cdot2\cdot2 + \tau\cdot\beta^2\cdot(1+\rho)\\
&= 4\sigma + 3\tau.
\end{align*}
Using this equation, we find the image of $f$, which is given in the third column of the following table.
\begin{center}
	\begin{tabular}{|c|c||c|}
	\hline
	$\sigma$ & $\tau$ & $a+b\rho \in \zrhomodunit{-6\rho}$\\
	\hline
	\hline
 1 & 1 & 1 \\ 
 1 & $\rho$ & $4+3\rho$ \\
 1 & $1+\rho$ & $1+3\rho$ \\ \hline
 $1+\rho$ & 1 & $1+4\rho$ \\
 $1+\rho$ & $\rho$ & $4+\rho$ \\
 $1+\rho$ & $1+\rho$ & $1+\rho$ \\ \hline
 $\rho$ & 1 & $3+4\rho$ \\
 $\rho$ & $\rho$ & $\rho$ \\
 $\rho$ & $1+\rho$ & $3+\rho$ \\ \hline
 2 & 1 & 5 \\
 2 & $\rho$ & $2+3\rho$ \\
 2 & $1+\rho$ & $5+3\rho$ \\ \hline
 $2\rho$ & 1 & $3+2\rho$ \\
 $2\rho$ & $\rho$ & $5\rho$ \\
 $2\rho$ & $1+\rho$ & $3+5\rho$ \\ \hline
 $2+2\rho$ & 1 & $5+2\rho$ \\
 $2+2\rho$ & $\rho$ & $2+5\rho$ \\
 $2+2\rho$ & $1+\rho$ & $5+5\rho$ \\
  \hline
	\end{tabular}
\end{center}

We now investigate the group structure of $\zrhomodunit{-6\rho}$. Let
\begin{align*}
H_1 &:= \langle[3+2\rho]\rangle = \{ 1, 3+2\rho, 5+2\rho, 5, 3+4\rho, 1+4\rho\},\\
H_2 &:= \langle[4+3\rho]\rangle = \{ 1, 4+3\rho, 1+3\rho\}.
\end{align*}
Observe that $H_1$ and $H_2$ are normal subgroups of $\zrhomodunit{-6\rho}$ and $H_1 \cap H_2 = \{ 1\}$ and $|H_1||H_2| = \Big|\zrhomodunit{-6\rho}\Big|$. Hence $$\zrhomodunit{-6\rho} = H_1 \times H_2 \cong \mathbb{Z}_6 \times \mathbb{Z}_3.$$
\end{example}


\subsection{Euler-Fermat theorem for the Eisenstein integers}

In 1640 in a letter to his colleague Fr\'{e}nicle de Bessy, Fermat writes without proof what is now known as \textit{Fermat's Little Theorem} which states that for a prime $p$ and an integer $a$ coprime to $p$ that $a^{p-1} \equiv 1 \pmod{p}$. This was later proven by Euler in 1736, but the generalization of this result sometimes referred to as the \textit{Euler-Fermat Theorem}, which states that for coprime positive integers $a$ and $n$ it follows that $a^{\varphi(n)} \equiv 1 \pmod{n}$, was proven by Euler in 1763~\cite{Euler1736,Euler1763}. It is this statement that we generalize to the Eisenstein integers.

\begin{theorem}\label{thm:Euler_Fermat_theorem}
Given $\theta,\eta \in \zrho-\{0\}$ such that $\gcd(\theta,\eta) = 1$, it follows that $$\theta^{\ephi(\eta)} \equiv 1 \pmod{\eta}.$$
\end{theorem}

\begin{proof}
Let $\theta,\eta \in \zrho-\{0\}$ with $\gcd(\theta,\eta) = 1$, and consider $\zrhomodunit{\eta}$. Since $\theta$ and $\eta$ are relatively prime, then $\theta$ is congruent to exactly one element in the set of complete residue classes of $\zrhomodunit{\eta}$. Consider the cyclic subgroup $\langle [\theta] \rangle$ of $\zrhomodunit{\eta}$. Since the group $\zrhomodunit{\eta}$ is finite, then its subgroup $\langle [\theta] \rangle$ is $\{ [\theta], [\theta]^2, \ldots, [\theta]^k \}$ where, for some least positive integer $k$, we have $\theta^k \equiv 1 \pmod{\eta}$. Recall that $\left|\zrhomodunit{\eta}\right| = \ephi(\eta)$. Hence by Lagrange's theorem $k$ divides $\ephi(\eta)$, so $k\cdot m = \ephi(\eta)$ for some $m \in \mathbb{Z}$. But that implies
\begin{align*}
\theta^{\ephi(\eta)} = \theta^{k \cdot m} = \left( \theta^k \right)^m 	&\equiv 1^m \pmod{\eta}\\
			&\equiv 1 \pmod{\eta}.
\end{align*}
Thus $\theta^{\ephi(\eta)} \equiv 1 \pmod{\eta}$ as desired.
\end{proof}

\begin{example}
Consider $2 - \rho$ and $\beta^2$ in $\zrho$. Given that $N(2 - \rho) = 7$, the reader can easily verify that $2 - \rho$ is a Category~3 prime by Proposition~\ref{prop:primes}. Since $2 - \rho$ does not divide $\beta^2$, then $\gcd(2 - \rho, \beta^2) = 1$. By Theorem~\ref{thm:Euler_Fermat_theorem}, we can conclude that $(2 - \rho)^{\ephi(\beta^2)} \equiv 1 \pmod{\beta^2}$. To verify this congruence, it suffices to show that $(2 - \rho)^6 \equiv 1 \pmod{\beta^2}$ since $\ephi(\beta^2) = 6$ by Theorem~\ref{thm:sizeOfUnits}. In this case it  is particularly easy to confirm that we are correct since $\beta^2$ and 3 are associates by Proposition~\ref{prop:betasim}, and hence $(2 - \rho)^6$ modulo $\beta^2$ coincides with $(2 - \rho)^6$ modulo 3. Observe $(2 - \rho)^6 = -323-360 \rho$ which is congruent to $1$ modulo $3$, as desired.
\end{example}


\subsection{The range of \texorpdfstring{$\ephi$}{the Euler phi function}}
It is well known that the values $\varphi(n)$ of the standard Euler phi function on the integers is even for all $n \geq 3$. It is natural to examine the range values of the Euler phi function $\ephi$ on the Eisenstein integers. It turns out that $\varphi$ and $\ephi$ share a very similar range.

\begin{theorem}\label{thm:range}
Let $\eta \in \zrho-\{0\}$. Then $\ephi(\eta)$ is an even integer unless $\eta$ is a unit or any associate of $2$.
\end{theorem}

\begin{proof}
By Remark~\ref{rem:ephi_on_1}, we know $\ephi$ equals 1 when applied to 1 or any of its five associates. By \eqref{thm:sizeOfUnits_2} of Theorem~\ref{thm:sizeOfUnits}, we know $\ephi(2) = 3$ and hence $\ephi(\eta) = 3$ for all associates $\eta$ of 2.

Let $\eta \in \zrho-\{0\}$ such that $\eta$ is neither a unit nor an associate of 2. Thus in its prime factorization, $\eta$ has a prime factor (up to associates) of the form
\begin{itemize}
\item (Category 1): $\beta^n$ for some $n \geq 1$, or
\item (Category 2): $2^n$ for some $n \geq 2$, or $p^n$ for some $n \geq 1$ where $p>2$ is prime in $\mathbb{Z}$ and $p \equiv 2 \pmod{3}$, or
\item (Category 3): $\psi^n$ for some $n \geq 1$ where $q$ is prime in $\mathbb{Z}$, the norm $N(\psi)$ equals $q$, and $q \equiv 1 \pmod{3}$.
\end{itemize}
Since $\ephi$ is multiplicative, it suffices to show that in each of the three cases above $\ephi$ yields an even value. To this end we utilize the three Equations~\eqref{thm:sizeOfUnits_1}, \eqref{thm:sizeOfUnits_2}, and \eqref{thm:sizeOfUnits_3} of Theorem~\ref{thm:sizeOfUnits}.

By Equation~\eqref{thm:sizeOfUnits_1}, it follows that $\ephi(\beta^n) = 3^n - 3^{n-1}$, and so $\ephi(\beta^n) = 3^{n-1}(3-1) = 2 \cdot 3^{n-1}$ is even and hence $\ephi(\beta^n)$ is even. By Equation~\eqref{thm:sizeOfUnits_2}, it follows that $\ephi{(2^n)} = 2^{2n}-2^{2n-2}$ which is a difference of two even numbers when $n \geq 2$, and hence $\ephi(2^n)$ is even if $n \geq 2$. Moreover for $p \neq 2$ prime in $\mathbb{Z}$ with $p \equiv 2 \pmod{3}$, Equation~\eqref{thm:sizeOfUnits_2} yields $\ephi{(p^n)} = p^{2n}-p^{2n-2}$ and so $\ephi(p^n) = p^{2n-2}(p^2-1)$. Since $p \neq 2$ and prime, we know $p$ is odd and hence $p^2 - 1$ is even. Thus it follows that $\ephi(p^n)$ is even. Lastly, given $\psi$ with $N(\psi) = q$ where $q$ is prime in $\mathbb{Z}$ and $q \equiv 1 \pmod{3}$, Equation~\eqref{thm:sizeOfUnits_3} yields $\ephi(\psi^n) = q^n - q^{n-1}$, and so $\ephi(\psi^n) = q^{n-1}(q-1)$. Since $q$ is prime with $q \equiv 1 \pmod{3}$, we know $q$ is odd and hence $q-1$ is even. Thus $\ephi(\psi^n)$ is even. The claim follows.
\end{proof}


\subsection{Relative primality of real and rho parts of powers of Category 3 primes}

This last subsection is not an application of the Euler phi function. However, it is a very interesting observation on the structure of the real and rho parts of every Category 3 prime in $\zrho$.

\begin{theorem}\label{thm:relprime}
Let $q \equiv 1 \pmod 3$ where $q = \psi \overline{\psi}$ and $\psi = a+b\rho$ is prime in $\zrho$. Then given $\psi^n = c+d\rho$, it holds that $(c,d) = 1$ for all $n \geq 1$.
\end{theorem}
\begin{proof}
We proceed by induction. Consider the case when $n=1$. Then $\psi^n = \psi = a+b\rho$ and $N(\psi) = q = a^2 + b^2 -ab$. Assume by way of contradiction that there exists $r \in \mathbb{Z}$ such that $r$ divides $a$ and $b$. Then $r$ divides $a^2 + b^2 -ab$ so that $r$ divides $q$. Since $q$ is prime in $\mathbb{Z}$, we know $r \in \{ 1, q \}$. Consider the case when $r = q$. Then $a = qj_1$ and $b = qj_2$ for some $j_1, j_2 \in \mathbb{Z}$. Observe that
\begin{align*}
q = a^2 + b^2 - ab &= q^2j_1^2 + q^2j_2^2 - q^2j_1j_2\\
&= q^2(j_1^2 + j_2^2 - j_1j_2).
\end{align*}
Thus $1 = q(j_1^2 + j_2^2 - j_1j_2)$ forces $q=1$, a contradiction. Hence $r = 1$ implies $(a,b)=1$ as desired.

Now assume $\psi^k = c+d\rho$ where $(c,d) = 1$ for some $k > 1$. Consider $\psi^{k+1}$ and observe
\begin{align*}
\psi^{k+1} &= (a+b\rho)^{k+1}\\
&= (a+b\rho)^k(a+b\rho)\\
&= (c+d\rho)(a+b\rho)\\
&= (ac-bd)+(ad+bc-bd)\rho.
\end{align*}
Let $C=ac-bd$ and $D=ad+bc-bd$. It suffices to show that $(C,D)=1$. Suppose by way of contradiction that there exists an integer $r$ such that $r$ divides $C$ and $D$. Since $q^{k+1} = N(\psi^{k+1}) = C^2 + D^2 - CD$, we know $r$ divides $C^2 + D^2 - CD$ implies $r$ divides $q^{k+1}$. Then $q$ prime in $\mathbb{Z}$ implies $r \in \{ 1, q^m \}$ for some $m \leq k+1$. Consider the case when $r = q^m$. Then $C = q^mj_1$ and $D = q^mj_2$ for some $j_1,j_2 \in \mathbb{Z}$. Observe that 
\begin{align*}
q^{k+1} = C^2 + D^2 - CD &= q^{2m}j_1 + q^{2m}j_2 - q^{2m}j_1j_2\\
&= q^{2m}(j_1 + j_2 - j_1j_2)\\
&= q^{2m}N(j_1+j_2\rho).
\end{align*}
We have two cases:

\noindent\textbf{CASE 1:} Assume that $2m \geq k+1$. Then $q^{2m-(k+1)}N(j_1+j_2\rho) = 1$ forces $q^{2m-(k+1)} = 1$, a contradiction.

\noindent\textbf{CASE 2:} Assume that $2m < k+1$. Then $N(j_1+j_2\rho)=q^{k+1-2m}$. Observe that
\begin{align*}
N(j_1+j_2\rho)=q^{k+1-2m} &\Rightarrow j_1+j_2\rho = \psi^{k+1-2m}\\
&\Rightarrow q^m(j_1+j_2\rho) = q^m\psi^{k+1-2m}\\
&\Rightarrow C+D\rho = q^m\psi^{k+1-2m}\\
&\Rightarrow \psi^{k+1} = q^m\psi^{k+1}\psi^{-2m}\\
&\Rightarrow \psi^{2m} = q^m\\
&\Rightarrow \psi^2 = q\\
&\Rightarrow \psi\psi = \psi\overline{\psi}\\
&\Rightarrow \psi=\overline{\psi}.
\end{align*}
Notice that $\psi=\overline{\psi}$ is a contradiction since $\psi$ and $\overline{\psi}$ are distinct nonassociates by definition. 

In both cases, $r = 1$ is forced. Hence $(C,D)=1$ as desired. Therefore the claim holds for all $n \geq 1$.
\end{proof}


\section{Open Questions}

\begin{question}
Are the (nonassociate) distinct pairs of category 3 primes $\psi$ and $\overline{\psi}$ always of the same odd class? Does the corresponding $q$ predict the odd class of $\psi$ and $\overline{\psi}$?
\end{question}

\begin{question}
In the $\mathbb{Z}$-setting, it is known that the Euler phi function $\phi$ takes on only even values. The even values $n$ such that the equation $\phi(x)=n$ has no solution $x$ are called \textit{nontotients}. The first few nontotients are
$$14, 26, 34, 38, 50, 62, 68, 74, 76, 86, 90, 94, 98, 114, \ldots$$
Since 1956, attempts have been made to find criteria for when $n$ is a nontotient however only sufficiency criteria were found~\cite{Schinzel1956,Spyropoulos1989} or necessary and sufficient criteria for $n$ of a certain form~\cite{Bateman-Selfridge1963}. Most recently in 1993, Mingzhi gave the following neccessary and sufficient criteria for $n$ to be a nontotient when $n$ is divisible by only one factor of 2~\cite{Mingzhi1993}.
\begin{theorem}
Let $n = 2 p_1^{\alpha_1} p_2^{\alpha_2} \cdots p_s^{\alpha_s}$ with $2<p_1<p_2<\cdots <p_s$, where the $p_i$ are primes. Then $n$ is a nontotient if and only if
\begin{enumerate}[(i)]
\item $n + 1$ is composite;
\item $p_s - 1 \neq 2 p_1^{\alpha_1} p_2^{\alpha_2} \cdots p_{s-1}^{\alpha_{s-1}}.$
\end{enumerate}
\end{theorem}
\noindent To our knowledge, a necessary and sufficient criteria for arbitrary $n$ remains an open problem in the $\mathbb{Z}$-setting.

In the $\zrho$-setting, we proved in Theorem~\ref{thm:range} that the Euler phi function $\ephi$ on $\zrho$ is even except on all units and the associates of 2. A natural question to ask is whether a necessary and sufficient criteria can be found for when $\eta \in \zrho$ is a nontotient?
\end{question}

\begin{question}
It is well-known that for most integers $n$, the multiplicative group $\left( \bigslant{\mathbb{Z}}{(n)} \right)^\times$ is not cyclic. In 1801, Gauss published a proof that this group is cyclic if and only if $n$ is of the form $2$, $4$, $p^k$, or $2p^k$, where $p$ is an odd prime~\cite{Gauss1801}. Gauss attributes Euler for coining the term \textit{primitive roots} as the generators of these cyclic groups. In the Gaussian integer setting, Cross gives a classification of all Gaussian integers that have primitive roots~\cite{Cross1983}. A natural question to ask is which Eisenstein integers have primitive roots?
\end{question}

\begin{question}
A beautiful but not so well-known result is Menon's identity, originally proven in 1965 by P. Kesava Menon~\cite{Menon1965}. For a given $n \in \mathbb{N}$, this identity establishes the following tantalizing relation between gcd sums and the values $\phi(n)$ and $\tau(n)$, where $\phi$ is the usual Euler-phi function and $\tau(n)$ counts the number of positive divisors of $n$:
$$ \sum_{\substack{1 \leq k \leq n \\ \gcd(k,n)=1}} \!\!\!\!\! \gcd(k-1, n) = \phi(n) \tau(n).$$
Can we generalize this captivating identity to the Eisenstein integer setting?
\end{question}


\end{document}